\documentclass{amsart}

\usepackage{hyperref}
\usepackage{epsfig}
\usepackage{appendix}
\usepackage{amsfonts,amssymb,amsmath,amsthm,amscd}
\usepackage{latexsym}
\usepackage{mathtools}
\usepackage{graphicx}
\usepackage[all, cmtip]{xy}
\usepackage{xcolor}
\usepackage{tikz-cd}
\usepackage{comment}
\usepackage[normalem]{ulem}

\newcommand{\A}{{\mathbb{A}}}

\newcommand{\C}{{\mathbb{C}}}

\newcommand{\G}{{\mathbb{G}}}
\renewcommand{\H}{{\mathbb{H}}}
\newcommand{\I}{{\mathbb{I}}}

\renewcommand{\L}{{\mathbb{L}}}

\newcommand{\N}{{\mathbb{N}}}

\newcommand{\Q}{{\mathbb{Q}}}
\newcommand{\R}{{\mathbb{R}}}

\newcommand{\Z}{{\mathbb{Z}}}

\newcommand{\Sb}{{\mathbf{S}}}

\newcommand{\Zb}{{\mathbf{Z}}}

\newcommand{\Dcal}{{\mathcal{D}}}
\newcommand{\Ecal}{{\mathcal{E}}}

\newcommand{\del}{\partial}

\newcommand{\id}{{\textup{id}}}

\newcommand{\num}{\textup{num}}
\newcommand{\Coh}{\textup{Coh}}

\newcommand{\fd}{\textup{fd}}
\newcommand{\rank}{\textup{rank}}

\DeclareMathOperator{\Aut}{Aut}

\DeclareMathOperator{\Cl}{Cl}

\DeclareMathOperator{\disc}{disc}

\DeclareMathOperator{\GL}{GL}

\DeclareMathOperator{\Hom}{Hom}

\DeclareMathOperator{\sheafhom}{\mathcal{H}\kern -.5pt \emph{om}}

\DeclareMathOperator{\Mat}{Mat}

\DeclareMathOperator{\Out}{\textup{Out}}

\DeclareMathOperator{\Rep}{Rep}

\DeclareMathOperator{\SL}{SL}
\DeclareMathOperator{\SO}{SO}
\DeclareMathOperator{\SU}{SU}

\DeclareMathOperator{\Sym}{Sym}

\DeclareMathOperator{\tr}{tr}

\newcommand{\git}{\mathbin{
  \mathchoice{/\mkern-6mu/}
    {/\mkern-6mu/}
    {/\mkern-5mu/}
    {/\mkern-5mu/}}}

\newcommand{\tl}{\triangleleft}

\theoremstyle{plain}
		\newtheorem{theorem}{Theorem}[section]
		\newtheorem{lemma}[theorem]{Lemma}
		\newtheorem{corollary}[theorem]{Corollary}
		\newtheorem{proposition}[theorem]{Proposition}

		\newtheorem{problem}[theorem]{Problem}

		\newtheorem*{nclaim}{Claim}
		
\theoremstyle{definition}
		\newtheorem{definition}[theorem]{Definition}
		\newtheorem{def-prop}[theorem]{Definition-proposition}

		\newtheorem{example}[theorem]{Example}
\theoremstyle{remark}
		\newtheorem*{remark}{Remark}
		
\DeclareMathOperator{\Ort}{O}
\DeclareMathOperator{\Pin}{Pin}
\DeclareMathOperator{\Spin}{Spin}
				
\begin{document}
\title[Surfaces, braids, Stokes matrices,
and points on spheres]{Surfaces, braids, Stokes matrices,\\
and points on spheres}

\author[Yu-Wei Fan]{Yu-Wei Fan}

\address{Department of Mathematics,
University of California, Berkeley,
Office: 735 Evans Hall,
Berkeley, CA 94720-3840, USA}

\email{ywfan@berkeley.edu}

\author[Junho Peter Whang]{Junho Peter Whang}

\address{Department of Mathematics,
Massachusetts Institute of Technology, 
Office: Building 2 Room 2-238a, 77 Massachusetts Avenue
Cambridge, MA 02139-4307, USA}

\email{jwhang@mit.edu}

\date{\today}

\maketitle

\begin{abstract}

Moduli spaces of points on $n$-spheres carry natural actions of braid groups. For $n=0$, $1$, and $3$, we prove that these symmetries extend to actions of mapping class groups of positive genus surfaces, by establishing exceptional isomorphisms with certain moduli of local systems. This relies on the existence of group structure for spheres in these dimensions. We also use the connection to demonstrate that the space of rank 4 Stokes matrices with fixed Coxeter invariant of nonzero discriminant contains only finitely many integral braid group orbits.
\end{abstract}

\setcounter{tocdepth}{1}
\tableofcontents

\section{Introduction} \label{sect:1}
\subsection{\unskip}
Let $S(m)\subset\A^m$ be the complex affine hypersurface $x_1^2+\cdots+x_m^2=1$. Define the \emph{moduli space of $r$ points on $S(m)$} to be
$$A(r,m)=S(m)^r\git\SO(m),$$
the geometric invariant theory quotient of $S(m)^r$ by the diagonal action of $\SO(m)$. The binary operation
$u\tl v=s_u(v)=2\langle u,v\rangle u -v$,
where $\langle\cdot,\cdot\rangle$ is the standard bilinear form on $\A^m$, equips $S(m)$ with the structure of a quandle (see Definition \ref{quandle} and Proposition \ref{spherequand}). The quandle structure on $S(m)$ endows $A(r,m)$ with a natural action of the braid group $B_r$ on $r$ strands. We define the \emph{Coxeter invariant} to be the morphism
$$c\colon A(r,m)\to\Pin(m)\git\SO(m)$$
mapping the class of each sequence $(u_1,\dots,u_r)\in S(m)^r$ to the class of the product $u_1\otimes\dots\otimes u_r\in\Pin(m)$ under the natural embedding of $S(m)$ in the pin group $\Pin(m)$. This construction and terminology are motivated by the notion of pseudo Coxeter element introduced in \cite[Section 3]{bries} in the general setting of quandles or racks. It can be observed (Proposition \ref{sphere-coxeter}) that the Coxeter invariant is invariant under the $B_r$-action on $A(r,m)$.

It is classical that the standard unit sphere $S^{m-1}\subset\R^m$ admits the structure of a topological group for $m=1,2,4$, as the norm on units in $\R,\C,\H$ respectively. In the first part of this paper, we use an algebraic version of this fact this to establish exceptional isomorphisms between moduli of points on $S(m)$ for $m=1,2,4$ and certain moduli spaces of local systems on surfaces. Let $\Sigma_{g,n}$ be a surface of genus $g$ with $n$ boundary curves. Given a complex reductive algebraic group $G$, let $$X(\Sigma_{g,n},G)=\Hom(\pi_1(\Sigma_{g,n}),G)\git G$$
be the coarse moduli space of $G$-local systems on $\Sigma_{g,n}$. It carries an action of the pure mapping class group $\Gamma_{g,n}$ of the surface. There is a $\Gamma_{g,n}$-invariant morphism
$$X(\Sigma_{g,n},G)\to (G\git G)^n$$
assigning to each $G$-local system the sequence of its monodromy classes along the boundary curves of $\Sigma_{g,n}$ with orientations inherited from $\Sigma_{g,n}$. If $n\in\{1,2\}$, the braid group $B_{2g+n}$ embeds into $\Gamma_{g,n}$ as a subgroup generated by Dehn twists along a suitable chain of simple loops in $\Sigma_{g,n}$ (see Section \ref{sect:4.2}), leading to an action of $B_r$ on the moduli space $X(\Sigma_{g,n},G)$.

\begin{theorem}
\label{mainthm}
Let $r\geq3$, and let $\Sigma_{g,n}$ be a surface of genus $g=\lfloor(r-1)/2\rfloor$ with $n=r-2g\in\{1,2\}$ boundary curves. We have $B_r$-equivariant isomorphisms of complex algebraic varieties
\begin{enumerate}
    \item $A'(r,1)\simeq X(\Sigma_{g,n},\mu_2)$,
    \item $A(r,2)\simeq X(\Sigma_{g,n},\G_m)$, and
    \item $A(r,4)\simeq X(\Sigma_{g,n},\SL_2)$
\end{enumerate}
where the Coxeter invariant of a point on the left hand side determines the boundary monodromy of the corresponding local system, and vice versa. In particular, the action of $B_r$ on the left hand side extends to an action of $\Sigma_{g,n}$.
\end{theorem}

Here, we define $A'(r,m)$ to be the quotient of $A(r,m)$ by the natural action of $\mu_2\simeq\Ort(m)/\SO(m)$. We will show that $A'(r,m)$ inherits the $B_r$-action and Coxeter invariant $c:A'(r,m)\to\Pin(m)\git\Ort(m)$. It is worth remarking that the definitions of $A(r,m)$, its braid group action, and Coxeter invariant have no \emph{a priori} connection to surfaces; nevertheless, topology of surfaces naturally emerges in the presence of exceptional structure on $S(m)$.

\begin{remark} The $7$-dimensional sphere $S^7\subset\R^8$, while not a topological group, carries the structure of a Moufang loop as the norm one units in the algebra of Octonions. It would be interesting to see whether this gives rise to additional structures on the moduli spaces $A_P(r,8)$. More generally, it would be desirable to understand the symmetries of the spaces $A_P(r,m)$ beyond the braid group action.

While we have stated Theorem \ref{mainthm} as isomorphisms of complex affine varieties, it can be observed that the isomorphisms can be made compatible with additional structures (e.g.~real structures) on the moduli spaces. For example, the quotient space $(S^3)^r/\SO(4,\R)$, where $S^3$ denotes the unit $3$-sphere, admits a $B_r$-equivariant isomorphism with the moduli space of $\SU(2)$-local systems on $\Sigma_{g,n}$, and hence admits an action of $\Gamma_{g,n}$.
\end{remark}

The second part of this paper concerns Stokes matrices. By definition, a \emph{Stokes matrix} of rank $r$ is an $r\times r$ unipotent upper triangular matrix. By the invariant theory of orthogonal groups, the quotient of the moduli space $A(r,r)$ by the natural action of $\mu_2\simeq\Ort(r)/\SO(r)$ can be identified with the moduli space $V(r)$ of rank $r$ Stokes matrices (see Section \ref{sect:5} for a more detailed discussion). This gives rise to a $B_r$-action on $V(r)$ and Coxeter invariant $c\colon V(r)\to\Pin(r)\git\Ort(r)$, which recover classical constructions on Stokes matrices. Under the natural map $\Pin(r)\git\Ort(r)\to\GL(r)\git\GL(r)$, the Coxeter invariant of a Stokes matrix is sent to the (conjugacy) class of the element $-s^{-1}s^T\in\GL(r)$ whose characteristic polynomial is reciprocal. Given a monic reciprocal polynomial $p$ of degree $r$, let us define
$$V_p(r)=\{s\in V(r):p(\lambda)=\det(\lambda+s^{-1}s^T)\}\subset V(r).$$
By the above remarks, each $V_p(r)$ is a $B_r$-invariant subvariety of $V(r)$. The following Diophantine problem appears to be fundamental.

\begin{problem}\label{mainproblem}
Understand the structure of the integral points on the varieties $V_p(r)$ under the braid group action and other symmetries.
\end{problem}

Here, a point on $V(r)$ is integral if the corresponding Stokes matrix has integral entries. In the simplest nontrivial case $r=3$, the Coxeter invariant associated to a Stokes matrix
$$s=\begin{bmatrix}1 & x & z\\ 0 & 1 & y\\ 0 & 0 & 1\end{bmatrix}$$
has characteristic polynomial $\det(\lambda+s^{-1}s^T)=(\lambda+1)(\lambda^2-k\lambda+1)$
where
$$k=x^2+y^2+z^2-xyz-2.$$
Thus, $V_p(3)$ is an affine cubic algebraic surface, and an argument going back to 1880 work of Markoff can be used to prove the following: if $\disc(p)\neq0$, then $V_p(3)$ contains at most finitely many integral $B_3$-orbits. Next, in the case $r=4$, given a Stokes matrix
\begin{align*}
s=
\begin{bmatrix}
1 & a & e & d\\
0 & 1 & b & f\\
0 & 0 & 1 & c\\
0 & 0 & 0 & 1
\end{bmatrix}
\end{align*}
its Coxeter invariant has characteristic polynomial
$$\det(\lambda+s^{-1}s^T)=\lambda^4-k_1k_2\lambda^3+(k_1^2+k_2^2-2)\lambda^2-k_1k_2\lambda+1$$
where $k_1$ and $k_2$ are given by
\begin{align*}
k_1+k_2&=ac+bd-ef\\
    k_1k_2&=a^2+b^2+c^2+d^2+e^2+f^2-abe-adf-bcf-cde+abcd-4.
\end{align*}
The discriminant of the above polynomial is $\Delta_k=(k_1^2-4)^2(k_2^2-4)^2(k_1^2-k_2^2)^2$. In \cite{dTdVdB}, the integral $B_r$-orbits on $V_p(4)$ were completely classified for $p(\lambda)=(\lambda-1)^4$. In this paper, we prove the following.

\begin{theorem}\label{mainthm2}
If $p(\lambda)\in\Z[\lambda]$ is a monic reciprocal polynomial of degree $4$ such that $\disc(p)\neq0$, then $V_p(4)$ contains at most finitely many integral $B_4$-orbits.
\end{theorem}

In fact, Theorem \ref{mainthm2} is a refinement, in the case $\disc(p)\neq0$, of a general structure theorem for the integral points on $V_p(4)$ for arbitrary $p$. Indeed, Theorem \ref{mainthm} allows us to view $V_p(4)$ as a finite union of moduli spaces $X_k(\Sigma_{1,2},\SL_2)$. This in turn allows us to import a general Diophantine structure theorem, proved by the second author \cite[Theorem 1.1]{Whang2}, which states that the integral points on a relative moduli space of $\SL_2$-local systems on an arbitrary surface $\Sigma_{g,n}$ belong to mapping class group orbits of finitely many points and degenerate subvarieties. The required refinement to prove Theorem \ref{mainthm2} involves analysis of degenerate integral points.

One source of integral Stokes matrices are exceptional collections in triangulated categories. In this context, Theorem \ref{mainthm2} gives restrictions on the mutation classes of Gram matrices for full exceptional collections of length $4$ in a triangulated category admitting a Serre functor. Finally, Theorem \ref{mainthm} also allows us to conceptually simplify aspects of the work of Chekhov--Mazzocco \cite{ChekhovMazzocco} on embeddings of Teichm\"uller spaces into moduli of Stokes matrices.

\subsection*{Organization of the paper} In Section \ref{sect:2}, we introduce the moduli spaces of points on spheres, braid group actions, and Coxeter invariants. In Section \ref{sect:3}, we recall the notion of core quandles for groups and prove isomorphisms between certain moduli spaces of points on spheres and certain quotient sets of product core quandles. In Section \ref{sect:4}, we introduce the moduli spaces of local systems, and prove Theorem \ref{mainthm} by way of isomorphisms obtained in Section \ref{sect:3}. In Section \ref{sect:5}, we discuss the relationship between the moduli of points on spheres and the spaces of Stokes matrices.
In Section \ref{sect:6}, we discuss Diophantine aspects of the varieties $V_p(4)$ and prove Theorem \ref{mainthm2}. Finally, in Section \ref{sect:7} we consider applications of the Diophantine theorem to the study of exceptional collections.

\section{Moduli of points on spheres} \label{sect:2}
In this section, we introduce the moduli spaces of points on spheres and discuss their braid group actions. The section is organized as follows. In Section \ref{sect:2.braid}, we recall the definitions braid groups and quandles, and the construction of braid group actions from quandles. In Section \ref{sect:2.clifford}, we record relevant background on Clifford algebras and pin groups. Finally, in Section \ref{sect:2.spheres}, we discuss the quandle structure on spheres, introduce the moduli spaces of points on spheres and their braid group actions, and define their Coxeter invariants.

\subsection{Braids and quandles}\label{sect:2.braid}
We begin by recalling Artin's presentation of braid groups, which we take as their definition.

\begin{definition}
Let $r\geq1$ be an integer. The \emph{braid group on $r$ strands} is the group $B_r$ defined by generators $\sigma_1,\dots,\sigma_{r-1}$ subject to the following relations:
\begin{enumerate}
    \item $\sigma_i\sigma_j=\sigma_j\sigma_i$ if $|i-j|\geq2$, and
    \item $\sigma_i\sigma_j\sigma_i=\sigma_j\sigma_i\sigma_j$ if $|i-j|=1$ (braid relation).
\end{enumerate}
We will refer to the elements $\sigma_1,\dots,\sigma_{r-1}$ as the \emph{standard generators} of $B_r$.
\end{definition}

We will introduce several examples of spaces with braid group actions. This can be streamlined by the following notion.

\begin{definition}
\label{quandle}
A \emph{quandle} is a pair $(X,\triangleleft)$ consisting of a set $X$ and a binary operation $\tl\colon X\times X\to X$ such that the following three conditions hold:
\begin{enumerate}
    \item For any $x\in X$, we have $x\tl x=x$.
    \item For any $x,z\in X$ there exists a unique $y\in X$ such that $x\tl y=z$.
    \item For any $x,y,z\in X$ we have
    $$x\tl(y\tl z)=(x\tl y)\tl (x\tl z).$$
\end{enumerate}
\end{definition}

A morphism $\varphi\colon(X,\tl )\to (X',\tl')$ of quandles is a map of sets $\varphi\colon X\to X'$ such that $\varphi(x\tl y)=\varphi(x)\tl '\varphi(y)$ for all $x,y\in X$. We shall denote by $\Aut(X,\tl )$ the group of automorphisms of the quandle $(X,\tl)$. 
The following classical observation shows that quandles give rise to numerous examples of spaces with braid group action.

\begin{proposition}
\label{bries}
Let $(X,\tl)$ be a quandle, and let $G$ be a group acting on $X$ via quandle automorphisms. Fix an integer $r\geq1$, and let $X^r/G$ denote the quotient set of $X^r$ by the diagonal action of $G$. There is a right action of $B_r$ on $X^r/G$, described in terms of the standard generators $\sigma_1,\dots,\sigma_{r-1}$ of $B_r$ as follows:
$$\sigma_{i}^*[x_1,\dots,x_r]=[x_1,\dots,x_{i-1},x_i\tl x_{i+1},x_i,x_{i+2},\dots,x_r]$$
for every $[x_1,\dots,x_r]\in X^r/G$.
\end{proposition}

\begin{proof} This is standard; see e.g.~\cite[Proposition 3.1]{bries}.
First, we claim that the moves
$$\sigma_{i}^*(x_1,\dots,x_r)=(x_1,\dots,x_{i-1},x_i\tl x_{i+1},x_i,x_{i+2},\dots,x_r)$$
together generate an action of $B_r$ on $X^r$. It follows by condition (2) of Definition \ref{quandle} that each $\sigma_i^*$ above is a bijection of $X^r$ onto itself. It is moreover obvious that $\sigma_i^*\sigma_j^*=\sigma_j^*\sigma_i^*$ if $|i-j|\geq2$. Hence, it only remains to check the braid relations. For this, we may restrict to the case $r=3$. Note that
\begin{align*}
    &\sigma_1^*\sigma_2^*\sigma_1^*(x_1,x_2,x_3)=((x_1\tl x_2)\tl (x_1\tl x_3),x_1\tl x_2,x_1),\quad\text{while}\\
    &\sigma_2^*\sigma_1^*\sigma_2(x_1,x_2,x_3)=(x_1\tl (x_2\tl x_3),x_1\tl x_2,x_1).
\end{align*}
Hence, the braid relations follow from condition (3) in Definition \ref{quandle}. We conclude the proof by observing that the $B_r$-action on $X^r$ sends $G$-orbits to $G$-orbits, and hence descends to the quotient $X^r/G$.
\end{proof}

\begin{remark}
The reader will notice that condition (1) in Definition \ref{quandle} was not used in the proof of Proposition \ref{bries}. Omitting this condition leads to the definition of a \emph{rack} or \emph{automorphic set} (see \cite{bries}), and braid group actions can indeed be introduced in this greater generality.

Let $X$ be a rack. In \cite[Section 3]{bries}, the notion of a \emph{pseudo Coxeter element} $c_x\in\Aut(X,\tl)$ associated to a sequence $x=(x_1,\dots,x_r)\in X^r$ was introduced. It is given as the composition of left translations:
$$c_x(u)=x_1\tl(x_2\tl\dots(x_r\tl u)\dots)),\quad\text{for all $u\in X$}.$$
This construction will motivate our definitions of Coxeter invariants for moduli spaces of points on spheres, to be given in Section \ref{sect:2.spheres}.
\end{remark}

\subsection{Clifford construction}
\label{sect:2.clifford} We now record some background on Clifford algebras needed for later parts of the paper. Fix a field $F$ of characteristic zero. Let $q$ be a nondegenerate quadratic form over an $F$-vector space $V$ of finite dimension. We will denote by $\langle-,-\rangle$ the symmetric bilinear pairing on $V$ associated to $q$, given by
$$\langle u,v\rangle=\frac{q(u+v)-q(u)-q(v)}{2},\quad u,v\in V.$$
Let $\Ort(q)$ (resp.~$\SO(q)$) denote the orthogonal group (resp.~special orthogonal group) of the quadratic form $q$. Let $\Cl(q)$ be the \emph{Clifford algebra} over $F$ associated to $(V,q)$. Namely, it is the associative $F$-algebra
$$\Cl(q)=T(V)/\langle v\otimes v-q(v),v\in V\rangle$$
where $T(V)=\bigoplus_{i=0}^\infty V^{\otimes i}$ the tensor algebra of $V$. To avoid potential confusion in later parts of the paper, we will write $\otimes$ to indicate the multiplication operation in $\Cl(q)$. There is a natural $\Z/2\Z$-grading $\Cl(q)=\Cl^0(q)\oplus\Cl^1(q)$ on the Clifford algebra induced from the $\Z$-grading on $T(V)$, and there is an obvious embedding $V\to\Cl^1(q)$. The underlying vector space of $\Cl(q)$ has dimension $2^{\dim_F (V)}$ over $F$.

By functoriality, morphisms of quadratic spaces extend uniquely to morphisms of Clifford algebras. In particular, the orthogonal group $\Ort(q)$ of the quadratic form $q$ acts on $\Cl(q)$ by algebra automorphisms. If $F$ is algebraically closed, then for each integer $m\geq1$ there is up to isomorphism a unique quadratic space $(V,q)$ of dimension $m$ with $q$ nondegenerate (e.g.~take $(V_m,q_m)=(\A^m,x_1^2+\dots+x_m^2)$). We will denote the resulting Clifford algebra by $\Cl(m)$ when there is no risk for confusion, and employ similar notations for related constructions, e.g.~$\Ort(m)=\Ort(q)$.

We now introduce a subgroup of units of the Clifford algebras, called pin groups. Let $S(q)\subset V$ be the affine hypersurface defined by the equation $q(v)=1$. The embedding $V\to\Cl(q)$ induces an embedding $S(q)\subset\Cl(q)^\times$ since $u^{\otimes 2}=q(u)=1$ for every $u\in S(q)$.

\begin{definition}
The \emph{pin group} $\Pin(q)$ is the closed algebraic subgroup of $\Cl(q)^\times$ over $F$ generated by $S(q)$. We write $\Pin(q)=\Pin^0(q)\sqcup\Pin^1(q)$ where we denote $\Pin^i(q)=\Pin(q)\cap\Cl^i(q)$. We define the \emph{spin group} by $\Spin(q)=\Pin^0(q)$.
\end{definition}

By functoriality of the Clifford construction, the natural action of $\Ort(q)$ on $V$ induces an action of $\Ort(q)$ on $\Pin(q)$ by group automorphisms. Given $u\in S(q)$, let us denote by $s_u$ the linear transformation of $V$ given by
$$s_u(v)=2\langle v,u\rangle u-v\quad\text{for $v\in V$.}$$
It is straightforward to check that $s_u\in\Ort(q)$ and $s_u\circ s_u=1$ for each $u\in S(q)$. The following result is elementary but useful.

\begin{proposition}
\label{basic}
For any $u\in S(q)$ and $v\in V$, we have $s_u(v)=u\otimes v\otimes u^{-1}$.
\end{proposition}

\begin{proof}
For each $u,v\in V\subset \Cl(q)$, we have
$$2\langle u,v\rangle=q(u+v)-q(u)-q(v)=(u+v)^{\otimes2}-u^{\otimes2}-v^{\otimes2}=u\otimes v+v\otimes u.$$
Thus, if $u\in S(q)$, then $u\otimes v\otimes u^{-1}=(2\langle u,v\rangle-v\otimes u)\otimes u^{-1}=2\langle v,u\rangle u-v$, which is the desired result.
\end{proof}

Let $\alpha\colon\Pin(q)\to\Pin(q)$ be the automorphism of $\Pin(q)$ induced by the negation automorphism $v\mapsto -v$ of $(V,q)$. Proposition \ref{basic} allows us to define the morphism $\pi\colon\Pin(q)\to\Ort(q)$ taking $g\in\Pin(q)$ to $\pi(g)\in\Ort(q)$ given by
$$\pi(g)(v)=\alpha(g)\otimes v\otimes g^{-1}\quad\text{for $v\in V$.}$$

\begin{corollary}
\label{basiccor}
We have the following.
\begin{enumerate}
    \item The morphism $\pi\colon\Pin(q)\to\Ort(q)$ is surjective, and $\pi(\Spin(q))=\SO(q)$.
    \item $S(q)\subset\Pin^1(q)$ is a Zariski closed conjugacy class in $\Pin(q)$.
\end{enumerate}
\end{corollary}

\begin{proof}
(1) The first claim follows from the fact that the orthogonal group $\Ort(q)$ is generated by (hyperplane) reflections by the Cartan--Dieudonn\'e theorem, and the observation that every reflection on $(V,q)$ is of the form $-s_u$ for some $u\in S(q)$. Since any reflection has determinant $-1$, the second claim follows.

(2) This follows from the fact the natural action of $\Ort(q)$ on $S(q)$ is transitive. Since $S(q)$ is Zariski closed in $\Cl(q)$, \emph{a fortiori} it is Zariski closed in $\Pin(q)$.
\end{proof}

We close this subsection by discussing the structure of $\Cl(m)$ over algebraically closed fields of characteristic zero for $m=1,2,4$.

\begin{example}
\label{cliffordex}
Let $F$ be an algebraically closed field of characteristic zero.
\begin{enumerate}
    \item Let $q_1(x)=x^2$ be defined on $V_1=F$. Then $S(q_1)=\mu_2=\{\pm1\}$ has group structure induced by usual multiplication on $F$. Consider the $\Z/2\Z$-graded $F$-algebra
    $$M_1=M_1^0\oplus M_1^1=F\oplus F\iota$$
    where $\iota^2=1$. The embedding $j\colon V_1\to M_1$ given by $x\mapsto x\iota$ satisfies $$j(x)^2=(x\iota)(x\iota)=x^2\iota^2=q_1(x)$$
    for every $x\in V_1$. By the universal property of Clifford algebras, this extends to a $\Z/2\Z$-graded $F$-algebra morphism $j\colon\Cl(q_1)\to M_1$ which is an isomorphism for dimension reasons since it is clearly surjective.
    \item Let $q_2(x,y)=xy$ be defined on $V_2=F^2$. Then $$S(q_2)=\G_m=\{(x,y):xy=1\}$$
    is the multiplicative group, with group structure induced by componentwise multiplication on $F^2$. Let us consider the $\Z/2\Z$-graded $F$-algebra
    $$M_2=M_2^0\oplus M_2^1=F^2\oplus F^2\iota$$
    where $\iota^2=1$ and $\iota(a,b)=(b,a)\iota$ for every $(a,b)\in F^2$. The embedding $j\colon V_2\to M_2$ given by $(x,y)\mapsto (x,y)\iota$ satisfies
    $$j(x,y)^2=(x,y)\iota(x,y)\iota=(x,y)(y,x)\iota^2=(xy,xy)=q_2(x,y)(1,1)$$
    for every $(x,y)\in V_2$. By the universal property of Clifford algebras, this extends to a $\Z/2\Z$-graded $F$-algebra morphism $j\colon\Cl(q_2)\to M_2$ which is an isomorphism for dimension reasons since it is clearly surjective.
    \item Let $q_4(x_{ij})=x_{11}x_{22}-x_{12}x_{21}=\det(x_{ij})$ be the determinant form on the space $V_4=\Mat_2$ of $2\times2$ matrices. Then $S(q_4)=\SL_2$ has a group structure induced by the matrix algebra structure on $\Mat_2$. Consider the $\Z/2\Z$-graded $F$-algebra
$$M_4=M_4^0\oplus M_4^1=\Mat_2^2\oplus\Mat_2^2\iota$$
where $\iota^2=1$ and $\iota(a,b)=(b,a)\iota$ for every $(a,b)\in\Mat_2^2$. Given a matrix $x\in\Mat_2$, we shall denote by $\bar x$ its adjugate, so that if
$$x=\begin{bmatrix}x_{11} &x_{12}\\ x_{21} & x_{22}\end{bmatrix}\quad\text{then}\quad\bar x=\begin{bmatrix}x_{22} & -x_{12}\\ -x_{21} & x_{11}\end{bmatrix}.$$
The embedding $j\colon V_4\to M_4$ given by $x\mapsto(x,\bar x)\iota$ satisfies
$$j(x)^2=(x,\bar x)\iota(x,\bar x)\iota=(x,\bar x)(\bar x,x)\iota^2=(x\bar x,\bar xx)=\det(x)(1,1)$$
for every $x\in\Mat_2$. By the universal property of Clifford algebras, this extends to a $\Z/2\Z$-graded $F$-algebra morphism $j\colon\Cl(q_4)\to M_4$, which is an isomorphism for dimension reasons. Indeed, it suffices to observe that $j\colon\Cl(q_4)\to M_4$ is surjective. For this note that, for any $a,b\in\SL_2$,
$$(a,\bar a)(b,\bar b)(\bar a\bar b,ba)=(aba^{-1}a^{-1},1).$$
Since $\SL_2$ is perfect, it follows that $j(\Cl(q_4))$ contains elements of the form $(x,1)$ and $(1,x)$ for any $x\in\SL_2$. It is easy to deduce from this that $j(\Cl(q_4))=M_4$ since it contains the ``standard'' basis vectors of $M_4$,~e.g.
$$\left(\begin{bmatrix}1 & 0 \\ 0 & 0\end{bmatrix},0\right)=\left(\begin{bmatrix}1 & 1 \\ -1 & 0\end{bmatrix},1\right)-\left(\begin{bmatrix}0 & 1 \\ -1 & 0\end{bmatrix},1\right).$$
\end{enumerate}
\end{example}

\subsection{Moduli of points on spheres} \label{sect:2.spheres}
Fix a field $F$ of characteristic zero, and let $q$ be a nondegenerate quadratic form over an $F$-vector space $V$. As in Section \ref{sect:2.clifford}, let $S(q)$ be the affine hypersurface in $V$ defined by $q(v)=1$. The natural action of $\Ort(q)$ on $V$ preserves $S(q)$. For each $u\in S(q)$, let $s_u\in\Ort(q)$ be given by
$s_u(v)=2\langle u,v\rangle u-v$ for $v\in V$ as before.

\begin{proposition}
\label{spherequand}
The variety $S(q)$ admits a quandle structure under the operation
$$u\tl v=s_u(v)\quad\text{for $u,v\in S(q)$.}$$
With respect to this, the group $\Ort(q)$ acts on $S(q)$ by quandle automorphisms. We shall refer to $(S(q),\tl)$ as the \emph{sphere quandle} associated to $S(q)$.
\end{proposition}

\begin{proof}
Note first that $u\tl u=s_u(u)=2\langle u,u\rangle u -u=u$ for every $u\in S(q)$. Next, since $s_u\circ s_u=\id_{V}$ for each $u\in S(q)$, the left translation $L_u=u\tl -$ on $S(q)$ is an automorphism of the variety $S(q)$. Finally, note that
$$u\tl(v\tl w)=s_u(2\langle v,w\rangle v-w)=2\langle s_u(v),s_u(w)\rangle s_u(v)-s_u(w)=(u\tl v)\tl (u\tl w)$$
for each $u,v,w\in S(q)$, since $s_u\in\Ort(q)$. This proves that $(S(q),\tl)$ is a quandle, as desired. To prove the second statement, we note that
$$(gu)\tl(gv)=s_{gu}({gv})=2\langle gu,gv\rangle gu-gv=g(2\langle u,v\rangle u-v)=gs_u(v)=g(u\tl v)$$
for every $u,v\in S(q)$ and $g\in\Ort(q)$. This gives the desired result.
\end{proof}

\begin{remark}
Let $G$ be a group, and let $S\subset G$ be a union of conjugacy classes. Then $S$ admits the structure of a quandle under the operation $u\tl v=uvu^{-1}$. Proposition \ref{basic} shows that the sphere quandle structure on $S(q)$ above agrees with the quandle structure of $S(q)$ as a conjugacy class in $\Pin(q)$.
\end{remark}

We now introduce the main objects of our study.

\begin{definition}
Let $r$ be a positive integer. The \emph{moduli space of $r$ points on $S(q)$} is the geometric invariant theory quotient
$$A(r,q)=S(q)^r\git\SO(q)$$
of $S(q)^r$ by the diagonal action of $\SO(q)$. Similarly, the \emph{moduli space of $r$ unoriented points on $S(q)$} is the geometric invariant theory quotient
$$A'(r,q)=S(q)^r\git\Ort(q)\simeq A(r,q)\git\mu_2$$
for the action of $\mu_2=\Ort(q)/\SO(q)$ on $A(r,q)$.
\end{definition}
The definition of $A(r,q)$ is functorial in the oriented quadratic space $(V,q)$. In particular, if the base field $F$ is algebraically closed, then up to isomorphism there is a unique moduli space of $r$ points on the sphere for a nondegenerate quadratic form in $m$ variables; we shall denote it by $A(r,m)=S(m)\git\SO(m)$ as in Section \ref{sect:1}. Similarly, we shall write $A'(r,m)$ for the corresponding moduli space of unoriented points on $S(m)$.

\begin{proposition}
\label{braidsquad}
The braid group $B_r$ acts on $A(r,q)$ and $A'(r,q)$ by the moves
$$\sigma_i(u_1,\dots,u_r)=(u_1,\dots,u_{i-1},s_{u_i}(u_{i+1}),u_i,u_{i+2},\dots,u_r)$$
where $\sigma_1,\dots,\sigma_{r-1}$ are the standard generators of $B_r$.
\end{proposition}

\begin{proof}
This follows by combining Propositions \ref{bries} and \ref{spherequand}.
\end{proof}

Our next step is to define the notion of Coxeter invariant for points in $A(r,q)$.

\begin{proposition}
\label{sphere-coxeter}
The morphism $c\colon S(q)^r\to\Pin(q)$ given by
$$(u_1,\dots,u_r)\mapsto u_1\otimes\cdots\otimes u_r$$
is $\Ort(q)$-equivariant and is $B_r$-invariant.
\end{proposition}

\begin{proof} As $\Ort(q)$ acts on $\Cl(q)$ by algebra automorphisms, $c$ is $\Ort(q)$-equivariant. For $B_r$-invariance, fix $u\in S(q)^r$. If $\sigma=\sigma_i$ is one of the standard generators of $B_r$, then
\begin{align*}
    c(\sigma_i^{*}u)&=u_1\otimes\cdots\otimes u_{i-1}\otimes s_{u_i}(u_{i+1})\otimes u_i\otimes u_{i+2}\otimes \cdots \otimes u_{r}\\
    &=u_1\otimes\cdots\otimes u_{i-1}\otimes (u_i\otimes u_{i+1}\otimes u_i^{-1})\otimes u_i\otimes u_{i+2}\otimes \cdots \otimes u_{r}\\
    &=u_1\otimes\cdots\otimes u_r=c(u)
\end{align*}
by Proposition \ref{basic}. This completes the proof.
\end{proof}

\begin{definition}
\label{sphere:coxeter}
The \emph{Coxeter invariant} of $A(r,q)$ is the $B_r$-invariant morphism
$$c\colon A(r,q)\to\Pin(q)\git\SO(q)$$
given by $[u_1,\dots,u_r]\mapsto [u_1\otimes\cdots\otimes u_r]$. Given each class $P$ in $\Pin(q)\git\SO(q)$, we shall denote the corresponding fiber of $c$ by $A_P(r,m)$. Similarly, the \emph{Coxeter invariant} of $A'(r,q)$ is the $B_r$-invariant morphism $c\colon A'(r,q)\to\Pin(q)\git\Ort(q)$ given by $[u_1,\dots,u_r]\mapsto[u_1\otimes\cdots\otimes u_r]$, and we denote its fibers by $A_P'(r,m)$.
\end{definition}

\begin{remark}
Under the projection $\pi\colon\Pin(q)\git\SO(q)\to\Ort(m)\git\SO(q)$ induced by the morphism $\pi\colon\Pin(q)\to\Ort(q)$ from Section \ref{sect:2.clifford}, we see that the image of $c(u)$ for $u=[u_1,\dots,u_r]\in A(r,m)$ the class of
$$(-1)^rs_{u_1}\circ\dots\circ s_{u_r}.$$
\end{remark}

\section{Core quandles} \label{sect:3}
\subsection{Core quandles} Let us first recall the notion of a core quandle \cite{joyce} associated to a group.

\begin{proposition}
\label{trquand}
Let $G$ be a group. The operation
$$u\tl v=uv^{-1}u$$
for $u,v\in G$ equips the set underlying $G$ with the structure of a quandle. We shall refer to $(G,\tl)$ as the \emph{core quandle} associated to $G$.
\end{proposition}

\begin{proof}
Clearly, $u\tl u=u$ for all $u\in G$. Since $u\tl (u\tl v)=v$ for every $u,v\in G$, it follows that the left translation $u\tl -$ is an involutive bijection on $G$. Moreover, for $u,v,w\in G$ we have
\begin{align*}
    u\tl (v\tl w)&=u(vw^{-1}v)^{-1}u=(uv^{-1}u)(u^{-1}wu^{-1})(uv^{-1}u)=(u\tl v)\tl (u\tl w).
\end{align*}
This proves that $(G,\tl)$ is a quandle, as desired.
\end{proof}

Our next step is to construct out of $G$ a group acting on its core quandle.

\begin{definition}
Given a group $G$, let $G^{[2]}$ denote the semidirect product
$$G^{[2]}=(G\times G)\rtimes C_2$$ where the cyclic group $C_2=\{1,\iota\}$ of order two acts on $G\times G$ by permutation.
\end{definition}

\begin{proposition}
\label{aut}
Let $G$ be a group, and with core quandle $(G,\tl)$.
\begin{enumerate}
    \item The group $C_2=\{1,\iota\}$ acts on $(G,\tl)$ by $\iota\cdot a=a^{-1}$ for every $a\in G$.
    \item The group $G\times G$ acts on $(G,\tl)$ by
    $(g,h)\cdot a=gah^{-1}$
    for $a,g,h\in G$.
\end{enumerate}
These two actions define an action of $G^{[2]}=(G\times G)\rtimes C_2$ on $(G,\tl)$.
\end{proposition}

\begin{proof}
For every $a,b\in G$, we observe that
$$(a\tl b)^{-1}=(ab^{-1}a)^{-1}=a^{-1}ba^{-1}=a^{-1}\tl b^{-1}.$$
This proves the first part of the proposition. Next, for every $g,h,a,b\in G$ we have
$$g(a\tl b)h^{-1}=gab^{-1}ah^{-1}=(gah^{-1})(hb^{-1}g^{-1})(gah^{-1})=(gah^{-1})\tl(gbh^{-1}).$$ This proves the second part. Now let $\phi\colon C_2\to\Aut(G,\tl)$ and $\psi\colon G\times G\to \Aut(G,\tl)$ be the associated morphisms of groups. To prove the last assertion, it suffices to show that, for every $(g,h)\in G\times G$, we have
$$\psi(\iota (g,h))=\phi(\iota)\psi(g,h)\phi(\iota^{-1}).$$
For each $a\in G$, we have
\begin{align*}
    &\psi(\iota (g,h))(a)=\psi(h,g)(a)=hag^{-1}\quad\text{while}\\
    &\phi(\iota)\psi(g,h)\phi(\iota^{-1})(a)=\phi(\iota)\psi(g,h)(a^{-1})=\phi(\iota)(ga^{-1}h^{-1})=hag^{-1},
\end{align*}
showing that the two sides agree, as desired.
\end{proof}

To fix ideas, let now $G$ be a reductive algebraic group over a field of characteristic zero. (The reader will notice that our constructions and arguments apply almost verbatim to arbitrary groups.) We define the family $B(r,G)$ of varieties equipped with braid group action and a notion of Coxeter invariants. Given an integer $r\geq1$, let us equip $G^r$ with the action of $G\times G$ given by
 $$(x,y)\cdot(a_1,\dots,a_r)=(xa_1y^{-1},\dots,xa_ry^{-1})$$
for $(x,y)\in G\times G$ and $(a_1,\dots,a_r)\in G^r$.

\begin{definition}
Let $G$ be a group, and let $r\geq1$ be an integer. Let $$B(r,G)=G^r\git(G\times G).$$
\end{definition}

\begin{corollary}
\label{braidond}
Let $G$ be a group, and let $r\geq1$ be an integer. The braid group $B_r$ acts on $B(r,G)$ by the moves
$$\sigma_i^*[a_1,\dots,a_r]=[a_1,\dots,a_{i-1},a_ia_{i+1}^{-1}a_i,a_i,a_{i+2},\dots,a_r]$$
where $\sigma_1,\dots,\sigma_{r-1}$ are the standard generators of $B_r$.
\end{corollary}

\begin{proof}
This follows by combining Propositions \ref{bries}, \ref{trquand}, and \ref{aut}.
\end{proof}

\begin{proposition}
\label{group-coxeter}
The map $c\colon G^r\to G^{[2]}$ given by
$$c(a_1,\dots,a_r)=(a_1,a_1^{-1})\iota\cdots (a_r,a_r^{-1})\iota$$
is $(G\times G)$-equivariant and $B_r$-invariant. Here, $G^{[2]}=(G\times G)\rtimes C_2$ is equipped with the subgroup conjugation action of $G\times G$.
\end{proposition}

\begin{proof}
For any $(x,y)\in G^2$ and $z\in G$, we have
$$(xzy^{-1},(xzy^{-1})^{-1})\iota=(x,y)(z,z^{-1})\iota(x,y)^{-1}.$$
It follows that $c(xay^{-1})=(x,y)c(a)(x,y)^{-1}$ for any $(x,y)\in G^2$ and $a\in G^r$, thus proving that $c$ is $(G\times G)$-equivariant. To prove $B_r$-invariance, it suffices to note that for any $a,b\in G$ we have
\begin{align*}
    (ab^{-1}a,(ab^{-1}a)^{-1})\iota(a,a^{-1})=(ab^{-1}a,a^{-1}ba^{-1})(a^{-1},a)\iota=(a,a^{-1})\iota(b,b^{-1}),
\end{align*}
so $c(\sigma_i^*u)=c(u)$ for $u\in G^r$ where $\sigma_i$ is any of the standard generators of $B_r$.
\end{proof}

\begin{definition}
\label{:coxeter}
The \emph{Coxeter invariant} on $B(r,G)$ is the $B_r$-invariant morphism
$$c\colon B(r,G)\to G^{[2]}\git(G\times G)$$
given by $[a_1,\dots,a_r]\mapsto[(a_1,a_1^{-1})\iota\cdots (a_r,a_r^{-1})\iota]$.
\end{definition}

It will be useful to record the following.

\begin{definition}
Let $W(G)=G\git G$ be the quotient of $G$ by itself under conjugation.
\end{definition}

\begin{proposition}
\label{conj}
Under the conjugation action of $G\times G$ on $G^{[2]}$, we have:
\begin{itemize}
    \item $(G\times G)\git(G\times G)\simeq W(G)^2$, and
    \item $(G\times G)\iota\git(G\times G)\simeq W(G)$.
\end{itemize}
Thus, we have an isomorphism $G^{[2]}\git(G\times G)\simeq W(G)^2\sqcup W(G)$.
\end{proposition}

\begin{proof}
The first claim is obvious. To prove the second claim, let $(d,e)\iota\in (G\times G)\iota$ be given. For any $(x,y)\in G\times G$, we have
$$(x,y)(d,e)\iota (x,y)^{-1}=(x,y)(d,e)(y,x)^{-1}\iota=(xdy^{-1},yex^{-1}).$$
Thus, the morphism $f\colon(G\times G)\iota/(G\times G)\to W(G)$ given by $(d,e)\mapsto de$ is well defined. Consider now the morphism $g\colon W(G)\to (G\times G)\iota/(G\times G)$ given by $c\mapsto (c,1)\iota$. This is well-defined, since
$$(xcx^{-1},1)\iota=(x,x)(c,1)\iota(x,x)^{-1}$$
for every $x\in G$. Now, clearly $f\circ g=\id_{W(G)}$. To see that $g\circ f=\id_{(G\times G)\iota/(G\times G)}$, it suffices to note that $(de,1)\iota=(1,e^{-1})(d,e)\iota (1,e)$ for every $(d,e)\in G^2$. This proves the second claim.
\end{proof}

\subsection{Exceptional isomorphisms}\label{sect:3.2}
In Section 2, we introduced the moduli spaces $A(r,m)$ of points on spheres for $r,m\geq1$. We will prove the following.

\begin{proposition}
\label{sporadic}
There exist $B_r$-equivariant isomorphisms
\begin{enumerate}
    \item $A'(r,1)\simeq B(r,\mu_2)$,
    \item $A(r,2)\simeq B(r,\G_m)$,
    \item $A(r,4)\simeq B(r,\SL_2)$
\end{enumerate}
suitably preserving the Coxeter invariants.
\end{proposition}

\begin{proof}
(1) Let $q_1(x)=x^2$ be defined over $V_1=F$ as in Example \ref{cliffordex}(1). Under the isomorphism $j\colon\Cl(q_1)\simeq M_1$ given therein we have
$$\Pin(q_1)=\Pin^0(q_1)\sqcup\Pin^1(q_1)=\mu_2\sqcup\mu_2\iota=\mu_2\times\{1,\iota\}.$$
Note that $\Ort(q_1)=\mu_2$ acts on $\Pin^0(q)$ trivially, while it acts on $\Pin^1(q_1)=\mu_2\iota$ by the multiplicative structure of $\mu_2$. Thus, $\Pin(q_1)\git\Ort(q)=\mu_2\sqcup\{*\}$. Under the isomorphism above, for each $u,v\in S(q_1)$ we have
$$j(s_u(v))=j(u\otimes v\otimes u^{-1})=(u\iota)( v\iota) (u\iota)^{-1}=v\iota=j(uv^{-1}u)$$
where the product $uv^{-1}u$ is taken with respect to the group structure on $\mu_2$. This shows that the sphere quandle structure on $S(q_1)=\mu_2$ agrees with the core quandle structure. Thus, we have a $B_r$-equivariant isomorphism
$$A'(r,q_1)=S(q_1)^r\git\Ort(q_1)=\mu_2^r\git\mu_2=\mu_2^r\git(\mu_2\times\mu_2)=B(r,\mu_2)$$
where the $\mu_2^r\git\mu_2$ is taken with respect to the diagonal multiplication of $\mu_2$ on $\mu_2^r$. It remains to show that the isomorphism above preserves the Coxeter invariants. Our goal is to show the commutativity of the following diagram:
$$\xymatrix{
A'(r,q_1) \ar[d]^c \ar[rrr]^{\sim} & & &B(r,\mu_2)\ar[d]^c\\
\Pin(q_1)\git\Ort(q_1)\ar@{=}[r]&\mu_2\sqcup\{*\} \ar[r] &\mu_2^2\sqcup\mu_2&\ar@{=}[l]\mu_2^{[2]}\git(\mu_2\times\mu_2)}$$
where the embedding $\mu_2\sqcup\{*\}\to\mu_2^2\sqcup\mu_2$ maps $x\in\mu_2$ to $(x,x^{-1})\in\mu_2^2$ and maps $*$ to $1\in\mu_2$. Given $u=[u_1,\dots,u_r]\in A'(r,1)$
its image under the map $A'(r,q_1)\xrightarrow{c}\Pin(q_1)\git\Ort(q_1)=\mu_2\sqcup\{*\}$ is
$$[(u_1\iota)\cdots(u_r\iota)]=\left\{\begin{array}{l l}u_1\cdots u_r\in\mu_2&\text{if $r$ is even}\\\ast&\text{if $r$ is odd.}\end{array}\right.$$
If $u=[u_1,\dots,u_r]\in B(r,\mu_2)$, then under the map $B(r,\mu_2)\xrightarrow{c}\mu_2^{[2]}\git\mu_2^2=\mu_2^2\sqcup\mu_2$ we have
$$[(u_1,u_1^{-1})\iota\cdots (u_r,u_r^{-1})\iota]=\left\{\begin{array}{l l}(u_1\cdots u_r,(u_1\cdots u_r)^{-1})&\text{if $r$ is even}\\
1&\text{if $r$ is odd.}\end{array}\right.$$
Thus, the isomorphism $A'(r,q_1)\simeq B(r,\mu_2)$ respects the Coxeter invariants.

(2) Let $q_2(x,y)=xy$ be defined over $V_2=F^2$ as in Example \ref{cliffordex}(2). Under the isomorphism $j\colon\Cl(q_2)\simeq M_2$ given therein we have
$$\Pin(q_2)=\Pin^0(q_2)\sqcup\Pin^1(q_2)=\G_m\sqcup\G_m\iota=\G_m\rtimes\{1,\iota\}$$
where $\iota(x,x^{-1})=(x^{-1},x)\iota$ for any $(x,x^{-1})\in\G_m$. Note that $\Spin(q_2)=\G_m$ acts on $\Spin(q_2)$ by trivial conjugation, while it acts on $\Pin^1(q_2)=\G_m\iota$ by
$$(x,x^{-1})(y,y^{-1})\iota (x^{-1},x)=(x^2y,x^{-2}y^{-1})$$
for $(x,x^{-1})\in\Spin(q_2)$ and $(y,y^{-1})\iota\in\Pin^1(q_2)$. Thus $\Pin(q_2)\git\SO(q_2)=\G_m\sqcup\{*\}$.
Under the isomorphism above, for each $u,v\in S(q_2)$ we have
$$j(s_u(v))=j(u\otimes v\otimes u^{-1})=(u\iota)(v\iota)(u\iota)^{-1}=(uv^{-1}u)\iota=j(uv^{-1}u)$$
where the product $uv^{-1}u$ is taken with respect to the group structure on $\G_m$. This shows that the sphere quandle structure on $S(q_2)=\G_m$ agrees with the core quandle structure. Thus, we have a $B_r$-equivariant isomorphism
$$A(r,q_2)=S(q_2)^r\git\SO(q_2)=\G_m^r\git\G_m\simeq\G_m^r\git(\G_m\times\G_m)=B(r,\G_m)$$
where the quotient $\G_m^r\git\G_m$ is taken with respect to the diagonal multiplication (by square) of $\G_m$ on $\G_m^2$. It remains to show that the isomorphism above preserves the Coxeter invariants. Our goal is to show the commutativity of the diagram:
$$\xymatrix{
A(r,q_2) \ar[d]^c \ar[rrr]^{\sim} & & &B(r,\G_m)\ar[d]^c\\
\Pin(q_2)\git\SO(q_2)\ar@{=}[r]&\G_m\sqcup\{*\} \ar[r] &\G_m^2\sqcup\G_m&\ar@{=}[l]\G_m^{[2]}\git(\G_m\times\G_m)}$$
where the embedding $\G_m\sqcup\{*\}\to\G_m^2\sqcup\G_m$ maps $x\in\G_m$ to $(x,x^{-1})\in\G_m^2$ and maps $*$ to $1\in\G_m$. Given $u=[u_1,\dots,u_r]\in A(r,q_2)$, its image under the map $A(r,q_2)\xrightarrow{c}\Pin(q_2)\git\SO(q_2)=\G_m\sqcup\{*\}$ is
$$[(u_1\iota)\cdots(u_r\iota)]=\left\{\begin{array}{l l}u_1u_2^{-1}\cdots u_{r-1}u_r^{-1}\in\G_m&\text{if $r$ is even}\\\ast&\text{if $r$ is odd.}\end{array}\right.$$
If $u=[u_1,\dots,u_r]\in B(r,\G_m)$, then $c(u)\in\G_m^{[2]}\git(\G_m\times\G_m)=\G_m^2\sqcup\G_m$ is
$$[(u_1,u_1^{-1})\iota\cdots (u_r,u_r^{-1})\iota]=\left\{\begin{array}{l l}(u_1u_2^{-1}\cdots u_{r-1}u_r^{-1},u_1^{-1}u_2\cdots u_{r-1}^{-1}u_r)&\text{if $r$ is even}\\
1&\text{if $r$ is odd.}\end{array}\right.$$
Thus, the isomorphism $A(r,q_2)\simeq B(r,\G_m)$ respects the Coxeter invariants.

(3) Let $q_4(x_{ij})=\det(x_{ij})$ be defined over $V_4=\Mat_2$ as in Example \ref{cliffordex}(3). Under the isomorphism $j\colon\Cl(q_4)\simeq M_4$ given therein we have
$$\Pin(q_4)=\Pin^0(q_4)\sqcup\Pin^1(q_4)=(\SL_2\times\SL_2)\sqcup(\SL_2\times\SL_2)\iota=\SL_2^{[2]}.$$
Note that $\Spin(q_4)=\SL_2\times\SL_2$ acts on $\SL_2^{[2]}$ by conjugation, and it follows by Proposition \ref{conj} that $$\Pin(q_4)\git\SO(q_4)=\Pin(q_4)\git\Spin(q_4)=W(\SL_2)^2\sqcup W(\SL_2).$$
Under the isomorphism above, for each $u,v\in S(q_4)$ we have
\begin{align*}
    j(s_u(v))&=j(u\otimes v\otimes u^{-1})\\
    &=((u,u^{-1})\iota)((v,v^{-1})\iota)((u,u^{-1})\iota)^{-1}\\
    &=(uv^{-1}u,u^{-1}vu^{-1})\iota=j(uv^{-1}u).
\end{align*}
where the products such as $uv^{-1}u$ are taken with respect to the group structure on $\SL_2$. This shows that the sphere quandle structure on $S(q_4)=\SL_2$ agrees with the core quandle structure. Thus, we have a $B_r$-equivariant isomorphism
$$A(r,q_4)=S(q_4)^r\git\SO(q_4)=\SL_2^r\git(\SL_2\times\SL_2)=B(r,\SL_2).$$
It remains to show that the isomorphism above preserves the Coxeter invariants. Our goal is to show the commutativity of the following diagram:
$$\xymatrix{
A(r,q_4) \ar[d]^c \ar[rr]^{\sim} &  &B(r,\SL_2)\ar[d]^c\\
\Pin(q_4)\git\SO(q_4)\ar@{=}[r]&W(\SL_2)^2\sqcup W(\SL_2) &\ar@{=}[l]\SL_2^{[2]}\git(\SL_2\times\SL_2).}$$
Given $u=[u_1,\dots,u_r]\in A(r,q_4)$, we have
$$c(u)=[(u_1,u_1^{-1})\iota\cdots(u_r,u_r^{-1})\iota]\in\Pin(q_4)\git\SO(q_4)=\SL_2^{[2]}\git\SL_2^2$$
which agrees with the image of $u$ under $A(r,q_4)\simeq B(r,\SL_2)\xrightarrow{c}\SL_2^{[2]}\git\SL_2^2$. Thus, the isomorphism $A(r,q_4)\simeq B(r,\SL_2)$ respects the Coxeter invariants.
\end{proof}

\section{Moduli of local systems}\label{sect:4}

Let $G$ be a reductive algebraic group over a field $F$ of characteristic zero. In the previous section, we defined the space $B(r,G)$ and showed that the core quandle structure on $G$ induces an action of the braid group $B_r$ on $B(r,G)$. We introduced the Coxeter invariant $B(r,G)\to G^{[2]}\git G^2$, which is a $B_r$-invariant morphism. As we shall see, the space $B(r,G)$ can be viewed as a coarse moduli space of $G$-local systems on a graph ($1$-dimensional finite simplicial complex).

This section is organized as follows. In Section \ref{sect:4.1}, we discuss moduli of $G$-local systems on graphs, and establish an isomorphism between $B(r,G)$ and a new space denoted $C(r,G)$. In Section \ref{sect:4.2}, we consider the moduli of $G$-local systems on surfaces, and prove that $C(r,G)$ is isomorphic to the moduli space $X(\Sigma_{g,n},G)$ of $G$-local systems on the surface $\Sigma_{g,n}$ of genus $g$ with $n$ punctures, where
$$g=\left\lfloor\frac{n-1}{2}\right\rfloor\quad\text{and}\quad r=n-2g\in\{1,2\}.$$
The isomorphisms established in Sections \ref{sect:4.1} and \ref{sect:4.2} will be $B_r$-equivariant and preserve invariants. Finally, in Section \ref{sect:4.3}, we combine this with the sporadic isomorphisms given in Section \ref{sect:3.2} to complete the proof of Theorem \ref{mainthm}.

\subsection{Graphs}\label{sect:4.1}
Let $\pi$ be a finitely generated group. Let $G$ be a reductive algebraic group over a field $F$ of characteristic zero. The \emph{$G$-representation variety} $\Rep(\pi,G)$ of $\pi$ is the affine scheme defined by the functor
$$A\mapsto\Hom(\pi, G(A))$$
for every $F$-algebra $A$. The \emph{$G$-character variety of $\pi$} is defined to be the invariant theoretic quotient
$$X(\pi,G)=\Hom(\pi,G)\git G=F[\Hom(\pi,G)]^G$$
of the $G$-representation variety with respect to the conjugation action of $G$. The group $\Out(\pi)$ of outer automorphisms of $\pi$ has a natural right action on $X(\pi,G)$ by pullback of representations.

\begin{definition}
Given a connected manifold or finite simplicial complex $M$, the \emph{$G$-character variety of $M$} is the $G$-character variety of its fundamental group: $$X(M,G)=X(\pi_1(M),G).$$
\end{definition}

The variety $X(M,G)$ is also the (coarse) moduli space of $G$-local systems on $M$, and we will also refer to it as such.

\begin{example}
As in Section 2, we shall denote by $W(G)=G\git G$ the quotient of $G$ by the conjugation action of $G$. If $S^1$ denotes the oriented circle, monodromy along the generator of $\pi_1(S^1)$ gives us an isomorphism
$$X(S^1,G)\simeq W(G).$$
\end{example}

Suppose that $M$ is a finite connected graph, i.e.~$1$-dimensional simplicial complex. We can give an explicit presentation for $X(M,G)$ as follows. Let $V(M)$ and $E(M)$ respectively denote the sets of vertices and edges of $M$. Let us equip $M$ with the structure of a quiver, so that each edge $e\in E(M)$ has a source vertex $s(e)$ and a target vertex $t(e)$ in $V(M)$. The quiver structure gives us an isomorphism
$$X(M,G)\simeq G^{E(M)}\git G^{V(M)}$$
where the group $G^{V(M)}$ acts on $G(^{E(M)})$ by
$$(g_v)\cdot (h_e)=(g_{t(e)}h(e)g_{s(e)}^{-1})$$
for every $(g_v)\in G^{V(M)}$ and $(h_e)\in G^{E(M)}$.

\begin{example}
(1) Let $M_r$ denote the graph with vertex set $V(M)=\{v_1,v_2\}$ of size $2$ and edge set $E(M)=\{e_1,\dots,e_r\}$ of $r$ such that each edge joins $v_1$ and $v_2$. Let us endow $M_r$ with a quiver structure so that $s(e_i)=v_1$ and $t(e_i)=v_2$ for all $i=1,\dots,r$. We then have
$$X(M_r,G)\simeq G^r\git G^2=B(r,G).$$
(2) Let now $N_r$ denote the graph with vertex set $V(M)=\{w\}$ and edge set $E(M)=\{d_1,\dots,d_{r-1}\}$ of size $r-1$ such that each edge joins $w$ to itself. Let us endow $N_r$ with a quiver structure. We then have
$$X(N_r,G)\simeq G^{r-1}\git G$$
where $G$ acts on $G^{r-1}$ via diagonal conjugation.
\end{example}

\begin{definition}
Let $G$ be a group, and let $r\geq2$ be an integer. Let
$$C(r,G)=G^{r-1}\git G$$
be the quotient of $G^{r-1}$ by diagonal conjugation action of $G$. \end{definition}

Consider continuous map $f\colon N_r\to M_r$ sending the vertex $w$ to $v_1$ and each edge $d_i$ to the concatenation of edges $e_i$ and $e_{i+1}$. Since $f$ is a homotopy equivalence, and it induces an isomorphism
$$B(r,G)\simeq X(M_r,G)\simeq X(N_r,G)\simeq C(r,G).$$
The following proposition gives an explicit description of this isomorphism.

\begin{proposition}
\label{cd}
Let $G$ be a group, and let $r\geq2$ be an integer. We have an isomorphism $\Phi\colon B(r,G)\to C(r,G)$ given at the level of representatives by $$\Phi(a_1,\dots,a_r)=(a_1a_2^{-1},\dots,a_{r-1}a_r^{-1}),$$
with the inverse isomorphism $\Psi\colon C(r,G)\to B(r,G)$ given by
$$\Psi(b_1,\dots,b_{r-1})\mapsto((b_1\cdots b_{r-1}),(b_2\cdots b_{r-1}),\dots,b_{r-1},1).$$
The right $B_r$-action induced via $\Phi$ on $C(r,G)$ is described in terms of the standard generators of $B_r$ as follows:
\begin{enumerate}
    \item We have $\sigma_1^*(b_1,\dots,b_{r-1})=(b_1,b_1b_2,b_3,\dots,b_{r-1}).$
    \item If $1<i<r-1$, we have
    $$\sigma_i^*(b_1,\dots,b_{r-1})=(b_1,\dots,b_{i-2},b_{i-1}b_i^{-1},b_i,b_ib_{i+1},b_{i+2},\dots,b_{r-1}).$$
    \item We have $\sigma_{r-1}^*(b_1,\dots,b_{r-1})=(b_1,\dots,b_{r-2}b_{r-1}^{-1},b_{r-1})$.
\end{enumerate}
\end{proposition}

\begin{proof}
It is easy to check that $\Phi$ and $\Psi$ are well-defined, and are mutual inverses of each other. The description of the induced braid group action on $C(r,G)$ follows directly from this and Corollary \ref{braidond}. For example, if $1<i<r-1$, we have
\begin{align*}
    &\sigma_i^*(b_1,\dots,b_{r-1})\\
    &=\Phi\sigma_i^*\Psi(b_1,\dots,b_{r-1})\\
    &=\Phi\sigma_i^*((b_1\cdots b_{r-1}),(b_2\cdots b_{r-1}),\cdots,b_{r-1},1)\\
    &=\Phi((b_1\cdots b_{r-1}),\dots,(b_{i-1}\cdots b_{r-1}),b_i(b_i\cdots b_{r-1}),(b_i\cdots b_{r-1}),\dots,b_{r-1},1)\\
    &=(b_1,\dots,b_{i-2},b_{i-1}b_i^{-1},b_i,b_ib_{i+1},b_{i+2},\dots,b_{r-1}).
\end{align*}
The expressions for $\sigma_1^*$ and $\sigma_{r-1}^*$ can be derived similarly.
\end{proof}

\begin{definition}
The \emph{Coxeter invariant} on $C(r,G)$ is the composition
$$c\colon C(r,G)\xrightarrow{\Psi}B(r,G)\xrightarrow{c}G^{[2]}\git(G\times G)$$
where $\Psi$ is given as in Proposition \ref{cd}.
\end{definition}

\begin{corollary}
\label{c-coxeter}
Let $G$ be a group. Given $b=(b_1,\dots,b_{r-1})\in C(r,G)$, its Coxeter invariant $c(b)$ is given explicitly by
$$c(b)=\left\{\begin{array}{l l}(b_1b_3\cdots b_{r-1},(b_1b_2\cdots b_{r-1})^{-1}(b_2b_4\cdots b_{r-2}))\in W(G)^2&\text{if $r$ is even}\\(b_1 b_3\cdots b_{r-2})(b_1b_2\cdots b_{r-1})^{-1}(b_2b_4\cdots b_{r-1})\in W(G)&\text{if $r$ is odd.}
\end{array}\right.$$
Here we use the identification $G^{[2]}\git(G\times G)\simeq W(G)^2\sqcup W(G)$ proved in Proposition~\ref{conj}.
\end{corollary}

\begin{proof}
If $r$ is even, then $c(b)\in(G\times G)\git(G\times G)\simeq W(G)^2$ with the first component in $W(G)^2$ given by
\[
(b_1b_2\cdots b_{r-1})(b_2b_3\cdots b_{r-1})^{-1}\cdots(b_{r-3}b_{r-2}b_{r-1})(b_{r-2}b_{r-1})^{-1}b_{r-1}=b_1b_3\cdots b_{r-1},
\]
and the second component given by
\begin{multline*}
(b_1b_2\cdots b_{r-1})^{-1}(b_2b_3\cdots b_{r-1})\cdots(b_{r-3}b_{r-2}b_{r-1})^{-1}(b_{r-2}b_{r-1})b_{r-1}^{-1}\\
=(b_1b_2\cdots b_{r-1})^{-1}(b_2b_4\cdots b_{r-2}).
\end{multline*}

If $r$ is odd, then $c(b)=(d(b),e(b))\iota\in(G\times G)\iota\git(G\times G)\simeq W(G)$, where
\[
d(b)=(b_1b_2\cdots b_{r-1})(b_2b_3\cdots b_{r-1})^{-1}\cdots(b_{r-2}b_{r-1})b_{r-1}^{-1}=b_1 b_3\cdots b_{r-2}
\]
and
\begin{align*}
e(b)&=(b_1b_2\cdots b_{r-1})^{-1}(b_2b_3\cdots b_{r-1})\cdots(b_{r-2}b_{r-1})^{-1}b_{r-1}\\
&=(b_1b_2\cdots b_{r-1})^{-1}(b_2b_4\cdots b_{r-1}).
\end{align*}
Recall that the identification $(G\times G)\iota\git(G\times G)\simeq W(G)$ in Proposition~\ref{conj} maps $(d(b),e(b))\iota$ to the element $d(b)e(b)\in W(G)$.
This completes our proof.
\end{proof}

\subsection{Surfaces}\label{sect:4.2}
Throughout this paper, by a \emph{surface} we shall mean a compact oriented manifold of real dimension $2$ with possibly nonempty boundary. Given a surface $\Sigma$, we shall endow its boundary curves with orientations consistent with the orientation of the surface. The inclusion $\del \Sigma\to \Sigma$ of the boundary curves into $\Sigma$ induces a morphism
$$c\colon X(\Sigma,G)\to W(G)^{\pi_0(\del \Sigma)}$$
from the $G$-character variety of $\Sigma$ to the product of $|\pi_0(\del\Sigma)|$ copies of $W(G)$, given by sending a representation $\rho$ to the sequence of its monodromy classes along the boundary curves of $\Sigma$.

Given a surface $\Sigma$, let $\Gamma(\Sigma)$ denote the pure mapping class group of $\Sigma$. It is the group of isotopy classes of orientation-preserving homeomorphisms of the surface fixing the punctures and boundary curves pointwise. Given a reductive algebraic group $G$ over a field of characteristic zero and a connected surface $\Sigma$, we have a natural right action of $\Gamma(\Sigma)$ on the character variety $X(\Sigma,G)$ by pullback of representations, factoring through the morphism
$$\Gamma(\Sigma)\to\Out(\pi_1(\Sigma)).$$
Given a surface $\Sigma$ and a simple closed curve $a\subset\Sigma$, we shall denote by $\tau_a\in\Gamma(\Sigma)$ the associated left Dehn twist along $a$.

Suppose now that $\Sigma_{g,n}$ is a connected surface of genus $g\geq1$ with $n\in\{1,2\}$ boundary curves. Let us implicitly fix a base point $x\in \Sigma_{g,n}$ on the interior of $\Sigma_{g,n}$. The fundamental group of $\Sigma_{g,n}$ is free of rank $2g+n-1$. We introduce a preferred sequence of free generators below, using a ribbon graph presentation of $\Sigma_{g,n}$.

\begin{definition}
A \emph{hyperelliptic sequence of generators} for $\pi_1(\Sigma_{g,n})$ is a sequence $$(\alpha_1,\dots,\alpha_{2g+n-1})$$
of simple based loops on $\Sigma_{g,n}$ arranged as shown in Figure \ref{fig1}.
\end{definition}

\begin{figure}[ht]
    \centering
    \includegraphics[scale=0.8]{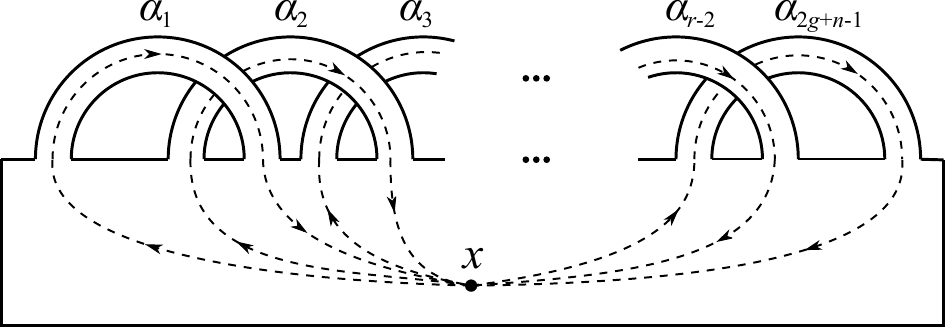}
    \caption{Hyperelliptic generators for $\pi_1(\Sigma_{g,n})$}
    \label{fig1}
\end{figure}

It is clear that the loops $\alpha_1,\dots,\alpha_{2g+n-1}$ together freely generate $\pi_1(\Sigma_{g,n})$. The above definition is motivated by the following.

\begin{proposition}
\label{braidprop}
Let $r\geq3$ be an integer, and write $r=2g+n$ with $n\in\{1,2\}$. Let $\Sigma_{g,n}$ be a surface of genus $g$ with $n$ boundary curves, and let $(\alpha_1,\dots,\alpha_{r-1})$ be a hyperelliptic sequence of generators for $\pi_1(\Sigma_{g,n})$. For $i=1,\dots,r-1$, let $\tau_i\in\Gamma(\Sigma_{g,n})$ be the left Dehn twist along the simple closed curve underlying $\alpha_i$.

\begin{enumerate}
    \item There is an embedding $B_{r}\to\Gamma(\Sigma_{g,n})$ sending $\sigma_i\mapsto\tau_i$ for $i=1,\dots,r-1$.
    \item The isomorphism
$$X(\Sigma_{g,n},G)\simeq C(r,G)$$
given by $\rho\mapsto(\rho(\alpha_1),\dots,\rho(\alpha_{r-1}))$ is $B_r$-equivariant for the $B_r$-action on $X(\Sigma_{g,n},G)$ induced by the braid group embedding $\textup{(1)}$ and the $B_r$-action on $C(r,G)$ given in Proposition \ref{cd}.
    \item The isomorphism in $\textup{(2)}$ gives rise to a commutative diagram
    $$\xymatrix{
X(\Sigma_{g,n},G) \ar[d]^k \ar[r]^{\sim} &C(r,G)\ar[d]^c\\
W(G)^n \ar@{=}[r] &W(G)^n}$$
where the vertical arrow on the left hand side sends $\rho\in X(\Sigma_{g,n},G)$ to its monodromy along the boundary curves, and where the vertical arrow on the right hand side is the morphism $c$ introduced in Corollary \ref{c-coxeter}.
\end{enumerate}
\end{proposition}

\begin{proof}
(1) The fact that the assignment $\sigma_i\mapsto\tau_i$ above gives rise to an embedding of $B_r$ into $\Gamma(\Sigma_{g,n})$ is classical and due to Birman and Hilden \cite{BirHil1, BirHil2}. 

(2) For convenience, we shall denote a simple loop lying in the same free homotopy class as the product of loops such as $\alpha_1\alpha_2$ by the same letters. We have the following.
\begin{itemize}
    \item $\tau_1(\alpha_2)=\alpha_1\alpha_2$ and $\tau_1(\alpha_i)=\alpha_i$ for all $i\neq2$.
    \item If $1<i<r-1$, then we have
    $$\left\{\begin{array}{l}\tau_i(\alpha_{i-1})=\alpha_{i-1}\alpha_i^{-1},\\\tau_i(\alpha_i)=\alpha_i,\quad\text{and}\\\tau_i(\alpha_{i+1})=\alpha_{i}\alpha_{i+1}\end{array}\right.$$
    while $\tau_i(\alpha_j)=\alpha_j$ whenever $|i-j|\geq2$.
    \item $\tau_{r-1}(\alpha_{r-2})=\alpha_{r-2}\alpha_{r-1}^{-1}$ and $\tau_{r-1}(\alpha_i)=\alpha_i$ for all $i\neq r-2$.
\end{itemize}
Combining this with the description of the braid group action on $C(r,G)$ given in Proposition \ref{cd}, we obtain the desired result.

(3) By Corollary \ref{c-coxeter}, if $b=(b_1,\dots,b_{r-1})\in C(r,G)$ then we have
\begin{itemize}
    \item $c(b)=(b_1b_3\cdots b_{r-1},(b_1b_2\cdots b_{r-1})^{-1}(b_2b_4\cdots b_{r-2}))$ if $r$ is even, while
    \item $c(b)=(b_1 b_3\cdots b_{r-2})(b_1b_2\cdots b_{r-1})^{-1}(b_2b_4\cdots b_{r-1})$ if $r$ is odd.
\end{itemize}
Then (3) follows simply by observing that the loop(s)
\begin{itemize}
    \item $\alpha_1\alpha_3\cdots \alpha_{r-1}$ and $(\alpha_1\alpha_2\cdots \alpha_{r-1})^{-1}(\alpha_2\alpha_4\cdots \alpha_{r-2})$ (if $r$ is even)
    \item $(\alpha_1 \alpha_3\cdots \alpha_{r-2})(\alpha_1\alpha_2\cdots \alpha_{r-1})^{-1}(\alpha_2\alpha_4\cdots \alpha_{r-1})$ (if $r$ is odd)
\end{itemize}
are freely homotopic to parametrizations of the boundary curve(s) of $\Sigma_{g,n}$. This can be readily seen by following along the boundary of the ribbon presentation of $\Sigma_{g,n}$, as shown in Figure \ref{fig1}.
\end{proof}

\begin{remark}
Note that Dehn twists along the boundary curves of $\Sigma$ act trivially on the moduli spaces $X(\Sigma_{g,n},G)$. By \cite[Section 4.4.4]{fm}, in the cases $r=3$ and $r=4$ the embedding $B_r\to \Gamma(\Sigma_{g,n})$ of the type given above $B_r$ isomorphically onto the pure mapping class group of the surface of genus $g$ with $n$ punctures. Hence, the braid group actions on the spaces $X(\Sigma_{1,1},G)$ and $X(\Sigma_{1,2},G)$ coincide with the pure mapping class group action.
\end{remark}

\subsection{Proof of Theorem \ref{mainthm}}\label{sect:4.3}
We now complete the proof Theorem \ref{mainthm}. We have seen from Propositions \ref{braidprop} and \ref{cd} that there is a $B_r$-equivariant isomorphism
$$X(\Sigma_{g,n},G)\simeq C(r,G)\simeq B(r,G)$$
that is compatible with invariants (boundary monodromy and Coxeter invariants). Combining this with the sporadic isomorphism established in Proposition \ref{sporadic} gives us the desired result.

\section{Stokes matrices}\label{sect:5}

In this section, we establish isomorphisms between moduli spaces of points on spheres and spaces of Stokes matrices. Together with Theorem~\ref{mainthm}, this gives a canonical embedding of the moduli of $\SL_2$-local systems on surfaces with two boundary components into the space of Stokes matrices. We also show the compatibility of Coxeter invariants, integral structures, and Poisson structures of this embedding.

\subsection{Stokes matrices and moduli of points on spheres}
\label{sect:5.1}
We begin with a definition.

\begin{definition}
A \emph{Stokes matrix} of rank $r$ is an $r\times r$ unipotent upper triangular matrix. We denote by $V(r)$ the affine space Stokes matrices of rank $r$.
\end{definition}

It is well-known that there is a braid group action on the space of Stokes matrices. Let $s\in V(r)$ and let $\sigma_1,\ldots,\sigma_{r-1}$ be the standard generators of $B_r$. The $B_r$-action on $V(r)$ is given by
\[
\sigma_is=
\begin{bmatrix}
\I_{i-1} & & & \\
& s_{i,i+1} & -1 & \\
& 1 & 0 & \\
& & & \I_{r-i-1}
\end{bmatrix}
\cdot s\cdot
\begin{bmatrix}
\I_{i-1} & & & \\
& s_{i,i+1} & 1 & \\
& -1 & 0 & \\
& & & \I_{r-i-1}
\end{bmatrix}.
\]
Note that the braid group action given above is slightly different from the one in \cite{Bondal04,Dubrovin}. It is clear that the characteristic polynomial of $-s^{-1}s^T$,
$$p(\lambda)=\det(\lambda+s^{-1}s^T),$$
is invariant under the $B_r$-action and satisfies $p(\lambda)=\lambda^rp(1/\lambda)$. The characteristic polynomial of $-s^{-1}s^T$ has been of great importance in the study of Stokes matrices. For instance, it is related to the monodromy data at infinity of certain Fuchsian systems \cite{Dubrovin,Ugaglia}; it also is related to the Serre functors of triangulated categories admitting full exceptional collections (see Section~\ref{sect:7}); finally, the eigenvalues of $-s^{-1}s^T$ are the Casimir functions of a natural Poisson bracket on the space of Stokes matrices \cite[Theorem~3.2]{Ugaglia}.

We relate the spaces of Stokes matrices with the moduli of points on spheres. We begin with introducing some notations. Let $\Sym(r)$ denote the affine space of $r\times r$ symmetric matrices. It is clear that $V(r)$ can be identified with the closed subscheme of $\Sym(r)$ consisting of symmetric matrices with $1$'s on the diagonal, via
\[
V(r)\hookrightarrow\Sym(r),\ \ \ s\mapsto\frac12\Big(s+s^T\Big).
\]
Let $\Sym(r,m)$ denote the closed subscheme of $\Sym(r)$ consisting of symmetric matrices of rank $\leq m$, and let $V(r,m)\subset V(r)$ be the preimage of $\Sym(r,m)$ under the embedding $V(r)\hookrightarrow\Sym(r)$. It is not hard to check that the closed subscheme $V(r,m)\subset V(r)$ is invariant under the $B_r$-action on $V(r)$.

\begin{proposition}
$(\A^m)^r\git\Ort(m)\simeq\Sym(r,m)$.
\end{proposition}

\begin{proof}
By the First Fundamental Theorem of Invariant theory for the orthogonal group (cf. \cite[Chapter 11 \textsection2.1]{Procesi} and \cite[Chapter 2 \textsection9]{Weyl}), any $\Ort(m)$-invariant polynomial function on $(\A^m)^r$ is a polynomial in the following functions
\[
\Phi_{ij}\colon(\A^m)^r\rightarrow \A^1, \ \ (v_1,\ldots,v_r)\mapsto\left<v_i,v_j\right>
\]
for $1\leq i,j\leq r$. Hence the quadratic map
\[
\Phi\colon(\A^m)^r\rightarrow\Sym(r), \ \ (v_1,\ldots,v_r)\mapsto(\left<v_i,v_j\right>)_{ij}
\]
descends to an embedding
\[
\widetilde\Phi\colon(\A^m)^r\git\Ort(m)\rightarrow\Sym(r).
\]
Thus it suffices to show that the image of $\Phi$ is $\Sym(r,m)\subset\Sym(r)$.
This is equivalent the following statement: for any $r\times r$ symmetric matrix $A$ of rank at most $m$, there exists $v_1,\ldots,v_r\in\A^m$ such that
\[
\begin{bmatrix}
- & v_1 & - \\
- & v_2 & - \\
 & \cdots &  \\
 - & v_r & - 
\end{bmatrix}
\begin{bmatrix}
\textemdash & \textemdash & & \textemdash\\
v_1&v_2&\cdots&v_r\\
\textemdash & \textemdash & & \textemdash
\end{bmatrix}
=A.
\]
This follows straightforwardly from the fact that for any symmetric matrix $A$, there exists an invertible matrix $P\in\GL_r(F)$ and a diagonal matrix $D$ such that $A=PDP^T$, if $\mathrm{char}(F)\neq2$.
\end{proof}


\begin{corollary}
\label{coro;A'andStokes}
Restricting the isomorphism $(\A^m)^r\git\Ort(m)\simeq\Sym(r,m)$ to the closed subscheme $V(r,m)\subset\Sym(r,m)$, one obtains a $B_r$-equivariant isomorphism
$$A'(r,m)= S(m)^r\git\Ort(m)\xrightarrow{\simeq}V(r,m)$$
given by
\[
[(v_1,\ldots,v_r)]\mapsto
\begin{bmatrix}
1 & 2\left<v_1,v_2\right> & \cdots & 2\left<v_1,v_r\right> \\
0 & 1 & \cdots & 2\left<v_2,v_r\right> \\
\vdots & \ddots & \ddots & \vdots\\
0 & \cdots & 0 & 1
\end{bmatrix}.
\]
\end{corollary}

\begin{proof}
The isomorphism $A'(r,m)\simeq V(r,m)$ follows from the previous proposition. Recall that the natural action of $B_r$ on $A'(r,m)=S(m)^r\git\Ort(m)$ is determined by the moves
\begin{align*}
\sigma_i(v_1,\ldots,v_r)&=(v_1,\ldots,v_{i-1},s_{v_i}(v_{i+1}),v_i,v_{i+2},\ldots,v_r)\\
&=(v_1,\ldots,v_{i-1},2\left<v_i,v_{i+1}\right>v_i-v_{i+1},v_i,v_{i+2},\ldots,v_r).
\end{align*}
It can be verified straightforwardly that this action is compatible with the action of $B_r$ on $V(r,m)$ via the isomorphism $S(m)^r\git\Ort(m)\simeq V(r,m)$ described explicitly above.
\end{proof}

Hence we have the following diagram:
\[
\begin{tikzcd}
S(1)^r \arrow[r,hook] \arrow[d] &
S(2)^r \arrow[r,hook] \arrow[d] &
\cdots \arrow[r,hook] &
S(r)^r \arrow[r,hook] \arrow[d] &
S(r+1)^r \arrow[r,hook] \arrow[d] &
\cdots\\
A(r,1) \arrow[r,hook] \arrow[d,"\deg 2"] &
A(r,2) \arrow[r,hook] \arrow[d,"\deg 2"] &
\cdots \arrow[r,hook] &
A(r,r) \arrow[r,"\deg 2"] \arrow[d,"\deg 2"] &
A(r,r+1) \arrow[r,"\simeq"] \arrow[d,"\simeq"] &
\cdots \\
A'(r,1) \arrow[r,hook] \arrow[d,"\simeq"] &
A'(r,2) \arrow[r,hook] \arrow[d,"\simeq"] &
\cdots \arrow[r,hook] &
A'(r,r) \arrow[r,"\simeq"] \arrow[d,"\simeq"] &
A'(r,r+1) \arrow[r,"\simeq"] \arrow[d,"\simeq"] &
\cdots \\
V(r,1) \arrow[r,hook] &
V(r,2) \arrow[r,hook] &
\cdots \arrow[r,hook] &
V(r) \arrow[r,equal] &
V(r) \arrow[r,equal] &
\cdots
\end{tikzcd}
\]

We define the Coxeter invariant on $V(r,m)$ by pulling back the Coxeter invariant on $A'(r,m)$ via the isomorphism $V(r,m)\simeq A'(r,m)$.

\begin{definition}
The \emph{Coxeter invariant} on $V(r,m)$ is defined to be the composition
\[
c\colon V(r,m)\xrightarrow{\sim}A'(r,m)\xrightarrow{c}\Pin(m)\git\Ort(m).
\]
It is a $B_r$-invariant morphism by Propositions~\ref{sphere-coxeter} and Corollary \ref{coro;A'andStokes}. We shall denote $V_P(r,m)=c^{-1}(P)$ for each $P\in\Pin(m)\git\Ort(m)$. Similarly, for each class $p\in\Ort(m)\git\Ort(m)$ we shall denote by $V_p(r,m)$ the subvariety of $V(r,m)$ consisting of Stokes matrices whose Coxeter invariant has class $p$ in $\Ort(m)\git\Ort(m)$.
\end{definition}

The next proposition shows that the Coxeter invariant of a rank $r$ Stokes matrix $s\in V(r)=V(r,r)$ gives a refinement of the characteristic polynomial of $-s^{-1}s^T$. Let $\textup{Poly}_r$ denote the space of polynomials of degree $r$.

\begin{proposition}
\label{proposition:coxeterStokes}
Let $f\colon\Ort(r)\git\Ort(r)\to\mathrm{Poly}_r$ be the map sending $[g]\in\Ort(r)\git\Ort(r)$ to the characteristic polynomial of $g$. Then the composition
\[
V(r)\xrightarrow{\sim}A'(r,r)\xrightarrow{c}\Pin(r)\git\Ort(r)\xrightarrow{\pi}\Ort(r)\git\Ort(r)\xrightarrow{f}\textup{Poly}_r
\]
takes a Stokes matrix $s\in V(r)$ to the characteristic polynomial of $-s^{-1}s^T$.
Here $\pi\colon\Pin(r)\rightarrow\Ort(r)$ is the surjection introduced in Section~\ref{sect:2.clifford}.
\end{proposition}

\begin{proof}
It suffices to prove the statement on a dense open subset of $V(r)$. We consider the subset $V(r)\backslash V(r,r-1)\subset V(r)$ consisting of Stokes matrices $s$ such that $\rank(s+s^T)=r$.
Observe that $s\in V(r)\backslash V(r,r-1)$ if and only if its image $[(v_1,\ldots,v_r)]\in A'(r,r)$ under the isomorphism $V(r)\simeq A'(r,r)$ satisfy the property that $\{v_1,\ldots,v_r\}\subset S(r)\subset\A^r$ is linearly independent.
With respect to the basis $\{v_1,\ldots,v_r\}$, the linear transformation $s_{v_i}\in\Ort(r)$ can be expressed as
\[
s_{v_i}=
\begin{bmatrix}
-1 & 0 & \cdots & 0 & \cdots & 0\\
0 & -1 & \cdots & 0 & \cdots & 0\\
\vdots & \vdots & \ddots & \vdots & \cdots & \vdots \\
2\left<v_1,v_i\right> & 2\left<v_2,v_i\right> & \cdots & 1 & \cdots & 2\left<v_n,v_i\right> \\
\vdots & \vdots & \cdots & \vdots & \ddots & \vdots \\
0 & 0 & \cdots & 0 & \cdots & -1
\end{bmatrix}.
\]
It is then easy to check the following Coxeter identity \cite{Bourbaki} holds
\[
s\cdot s_{v_1}\cdot\cdots\cdot s_{v_r}=(-1)^{r+1}s^T,
\]
where
\[
s=
\begin{bmatrix}
1 & 2\left<v_1,v_2\right> & \cdots & 2\left<v_1,v_n\right> \\
0 & 1 & \cdots & 2\left<v_2,v_n\right> \\
\vdots & \ddots & \ddots & \vdots\\
0 & \cdots & 0 & 1
\end{bmatrix}.
\]
The proposition then follows from the fact that the image of $[(v_1,\ldots,v_r)]\in A'(r,r)$ under the composition $$A'(r,r)\xrightarrow{c}\Pin(r)\git\Ort(r)\xrightarrow{\pi}\Ort(r)\git\Ort(r)$$
is given by $[(-1)^rs_{v_1}\circ\cdots\circ s_{v_r}]$.
\end{proof}

\begin{remark}
Given a Stokes matrix $s\in V(r)$, its Coxeter invariant $c(s)\in\Pin(r)\git\Ort(r)$ carries more information than the characteristic polynomial of $-s^{-1}s^T$ in general, as we shall see in Example \ref{r4example} below.
\end{remark}

\begin{example}
Let $s$ be a $3\times3$ Stokes matrix
$$s=\begin{bmatrix}1 & x & z\\ 0 & 1 & y\\ 0 & 0 & 1\end{bmatrix}.$$
The characteristic polynomial of its associated Coxeter element $-s^{-1}s^T$ is given by
\[
p(\lambda)=(\lambda+1)(\lambda^2-k\lambda+1),
\]
where $k=x^2+y^2+z^2-xyz-2$.
\end{example}

\begin{example}\label{r4example}
Let $s$ be a $4\times 4$ Stokes matrix
$$s=\begin{bmatrix}1 & a & e & d\\ 0 & 1 & b & f\\ 0 & 0 & 1 & c\\ 0 & 0 & 0 & 1\end{bmatrix}.$$
The characteristic polynomial of its associated Coxeter element $-s^{-1}s^T$ is given by
$$p_{k_1,k_2}(\lambda)=\lambda^4-k_1k_2\lambda^3+(k_1^2+k_2^2-2)\lambda^2-k_1k_2\lambda+1$$
where we have
\begin{align*}
k_1+k_2&=ac+bd-ef\\
k_1k_2&=a^2+b^2+c^2+d^2+e^2+f^2-abe-adf-bcf-cde+abcd-4.
\end{align*}
Note that the assignment $(k_1,k_2)\mapsto p_{k_1,k_2}$ is not injective. Indeed, we have
\[
p_{k_2,k_1}=p_{k_1,k_2} \text{ \ and \ } p_{-k_1,-k_2}=p_{k_1,k_2},
\]
which reflects the fact that the morphisms
\[
\Spin(4)\git\SO(4)\to\Spin(4)\git\Ort(4)  \text{ \ and \ } \Spin(4)\git\Ort(4)\to\SO(4)\git\Ort(4)
\]
respectively are generically finite of degree $2$.
This shows that the Coxeter invariant $c(s)\in\Pin(4)\git\Ort(4)$ carries more information than the characteristic polynomial $p$.
\end{example}

\begin{remark}
As we shall see in Section~\ref{sect:6}, the expressions for $k$ and $(k_1,k_2)$ in the above examples give equations for the moduli spaces of $\SL_2$-local systems with fixed boundatry traces on a one-holed torus and a two-holed torus, respectively. This relationship will be generalized in the next subsection.
\end{remark}

\subsection{Stokes matrices and character varieties}

We now establish the connection between the $\SL_2$-character varieties and the space of Stokes matrices, through the isomorphism $A'(r,m)\simeq V(r,m)$ in Corollary~\ref{coro;A'andStokes} and $A_P(r,4)\simeq X_k(\Sigma_{g,n},\SL_2)$ in Theorem~\ref{mainthm}.

First, we need to compare the invariant subvarieties $A_P(r,4)\subset A(r,4)$ and $A'_{P'}(r,4)\subset A'(r,4)$. Here $P\in\Pin(4)\git\SO(4)$ and $P'$ denotes its class in $\Pin(4)\git\Ort(4)$.
Recall that in the proof of Proposition~\ref{sporadic}, we have
$$\Pin(4)\git\Ort(4)=(\Pin(4)\git\SO(4))\git\{1,\iota\}=(W(\SL_2)^2\sqcup W(\SL_2))/\{1,\iota\}$$
where element $\iota$ acts on $\Pin^0(4)\git\SO(4)=W(\SL_2)^2$ by interchanging the components of $W(\SL_2)^2$ and acts on $\Pin^1(4)\git\SO(4)=W(\SL_2)$ trivially. It follows that $\Pin^0(4)\git\Ort(4)$ parametrizes unordered pairs of (not necessarily distinct) elements in $W(\SL_2)$. This shows the following.

\begin{proposition}
\label{prop:compareAA'}
Let $r\geq1$, and let $r_0\in\{0,1\}$ denote its remainder modulo $2$. Given $P\in\Pin^{r_0}(4)\git\SO(4)$, we have the following.
\begin{itemize}
    \item If $r= 3$, then the action of $\mu_2$ on $A(r,4)$ is trivial $A_P(3,4)\simeq A'_{P'}(3,4)$.
    \item If $r\geq 5$ is odd, then the involutive action of $\mu_2$ on $A(r,4)$ preserves $A_P(r,4)$, and the projection $A(r,4)\to A'(r,4)$ induces an isomorphism
    $$A_P(r,4)\git\mu_2\simeq A'_{P'}(r,4).$$
    \item If $r$ is even, then the involution in $\mu_2$ provides an isomorphism of $A_P(r,4)$ with $A_{\iota P}(r,4)$. The preimage of $A'_{P'}(r,4)$ along the projection $A(r,4)\to A'(r,4)$ is $A_P(r,4)\sqcup A_{\iota P}(r,4)$, and each component maps isomorphically onto $A'_{P'}(r,4)$.
\end{itemize}
\end{proposition}

We can now relate the $\SL_2$-character varieties with the space of Stokes matrices and compare their Coxeter invariants.

\begin{proposition}
\label{prop:chekhovmazzocco}
Let $r\geq3$ be an integer, and write $r=2g+n$ with $n\in\{1,2\}$. Let $k\in\A^n$. There is a canonical morphism from the moduli of local systems $X_k(\Sigma_{g,n},\SL_2)$ into the space of Stokes matrices $V(r)$ given by
\[
X_k(\Sigma_{g,n},\SL_2)\simeq A_P(r,4)\rightarrow A'_{P'}(r,4)\simeq V_{P'}(r,4)
\]
where the Coxeter invariant $P\in \Pin(r)\git\SO(r)$ is determined by $k$, and $P'$ is the class of $P$ in $\Pin(r)\git\Ort(r)$. If $r$ is even, the second arrow above is an isomorphism. Moreover, the corresponding class of $P'$ in $\Ort(r)\git\Ort(r)$ has characteristic polynomial $p$ described as follows:
    \begin{enumerate}
        \item If $r$ is odd, we have
            \[
            p(\lambda)=(\lambda^2-k\lambda+1)(\lambda+1)(\lambda-1)^{r-3}.
            \]
        \item If $r$ is even, we have
            \[
            p(\lambda)=\Big(\lambda^4-k_1k_2\lambda^3+(k_1^2+k_2^2-2)\lambda^2-k_1k_2\lambda+1\Big)(\lambda-1)^{r-4}.
            \]
    \end{enumerate}
\end{proposition}

\begin{proof}
The first part of the proposition follows from Theorem \ref{mainthm} and Corollary \ref{coro;A'andStokes} as well as the observation that, if $r$ is even, then $A_P(r,4)\rightarrow A'_{P'}(r,4)$ is an isomorphism by Proposition~\ref{prop:compareAA'}.

Now we relate the Coxeter invariants on both sides.
Suppose $[(v_1,\ldots,v_r)]\in A(r,4)$ and $\rho\in X(\Sigma_{g,n},\SL_2)$ are identified through the isomorphism established in Theorem~\ref{mainthm}. By Propositions~\ref{cd} and \ref{braidprop}, the Coxeter invariants of $\rho$, i.e.~the monodromy along the boundary curve(s), are described as follows.
\begin{itemize}
    \item When $r$ is odd, the monodromy of $\rho$ along the boundary curve is given by
    \[
    v_1v_2^{-1}\cdots v_rv_1^{-1}v_2\cdots v_r^{-1}\in\SL_2.
    \]
    \item When $r$ is even, the monodromy of $\rho$ along the two boundary curves are given by
    \[
    v_1v_2^{-1}\cdots v_r^{-1}\text{ \ and \ }v_1^{-1}v_2v_3^{-1}\cdots v_{r-1}^{-1}v_r\in\SL_2.
    \]
\end{itemize}

By Proposition~\ref{proposition:coxeterStokes}, the image of $[(v_1,\ldots,v_r)]\in A(r,4)$ under the composition $A(r,4)\rightarrow A'(r,4)\simeq V(r,4)\hookrightarrow V(r)$ lies in the closed subvariety $V_p(r)\subset V(r)$, where $p$ is the characteristic polynomial of
\[
(-1)^rs_{v_1}\circ\cdots\circ s_{v_r}\in\Ort(r).
\]
Here we regard each $v_i$ as an element in $S(r)$ by the embedding $S(4)\hookrightarrow S(r)$ to the first four components.
It is clear that the linear transformation $(-1)^rs_{v_1}\circ\cdots\circ s_{v_r}$ acts trivially on the last $r-4$ components. Hence it suffices to compute the characteristic polynomial of $(-1)^rs_{v_1}\circ\cdots\circ s_{v_r}$ as an element in $\Ort(4)$, and express it in terms of the boundary monodromy of $\rho$.

Recall that we have $s_v(u)=vu^{-1}v$, where $u,v$ are regarded as elements in $S(4)$ on the left hand side, and regarded as elements in $\SL_2$ on the right hand side. Therefore,
\begin{itemize}
    \item when $r$ is odd, $(-1)^rs_{v_1}\circ\cdots\circ s_{v_r}$ acts on $\SL_2$ as:
    \[
    u\mapsto-v_1v_2^{-1}v_3\cdots v_r u^{-1}v_rv_{r-1}^{-1}\cdots v_1;
    \]
    \item when $r$ is even, $(-1)^rs_{v_1}\circ\cdots\circ s_{v_r}$ acts on $\SL_2$ as:
    \[
    u\mapsto v_1v_2^{-1}\cdots v_r^{-1}uv_r^{-1}\cdots v_1.
    \]
\end{itemize}
It remains to prove the following linear algebraic lemma.
\end{proof}

\begin{lemma}
Let $a,b\in\SL_2$.
\begin{enumerate}
    \item\label{lemma:oddcase}
    The linear transformation $\Mat_{2}\rightarrow\Mat_{2}$ given by
    \[
    u\mapsto a\bar ub
    \]
    has characteristic polynomial
    \[
    p(\lambda)=(\lambda^2-k\lambda+1)(\lambda+1)(\lambda-1),
    \]
    where $k=-\tr(ab^{-1})$. (Recall that $\bar u$ denotes the adjugate matrix of $u$.)
    \item\label{lemma:evencase}
    The linear transformation $\Mat_{2}\rightarrow\Mat_{2}$ given by
    \[
    u\mapsto aub
    \]
    has characteristic polynomial
    \[
    p(\lambda)=\lambda^4-k_1k_2\lambda^3+(k_1^2+k_2^2-2)\lambda^2-k_1k_2\lambda+1,
    \]
    where $k_1=\tr(a)$ and $k_2=\tr(b)$.
\end{enumerate}
\end{lemma}

\begin{proof}
Both statements can be verified by direct computations. Let $a=[a_{ij}]_{1\leq i,j\leq2}$ and $b=[b_{ij}]_{1\leq i,j\leq2}$. The transformation in (\ref{lemma:oddcase}) can be represented by the matrix
\[
\begin{bmatrix}
a_{12}b_{21} & a_{12}b_{22} & a_{22}b_{21} & a_{22}b_{22} \\
-a_{11}b_{21} & -a_{11}b_{22} & -a_{21}b_{21} & -a_{21}b_{22} \\
-a_{12}b_{11} & -a_{12}b_{12} & -a_{22}b_{11} & -a_{22}b_{12} \\
a_{11}b_{11} & a_{11}b_{12} & a_{21}b_{11} & a_{21}b_{12}
\end{bmatrix}.
\]
Using the condition that $a,b\in\SL_2$, one can show that the eigenvalues of the above matrix are $\pm1$ and $\mu_1,\mu_2$, where $\mu_1,\mu_2$ are the eigenvalues of $-ab^{-1}$. This proves the statement in (\ref{lemma:oddcase}).

Similarly, the transformation in (\ref{lemma:evencase}) can be represented by the matrix
\[
\begin{bmatrix}
a_{11}b_{11} & a_{11}b_{12} & a_{21}b_{11} & a_{21}b_{12} \\
a_{11}b_{21} & a_{11}b_{22} & a_{21}b_{21} & a_{21}b_{22} \\
a_{12}b_{11} & a_{12}b_{12} & a_{22}b_{11} & a_{22}b_{12} \\
a_{12}b_{21} & a_{12}b_{22} & a_{22}b_{21} & a_{22}b_{22}
\end{bmatrix}.
\]
The eigenvalues of the matrix are $\mu_1\nu_1,\mu_1\nu_2,\mu_2\nu_1,\mu_2\nu_2$, where $\mu_1,\mu_2$ are the eigenvalues of $a$ and $\nu_1,\nu_2$ are the eigenvalues of $b$. The statement in (\ref{lemma:evencase}) then follows from
\[
\lambda^4-k_1k_2\lambda^3+(k_1^2+k_2^2-2)\lambda^2-k_1k_2\lambda+1
=(\lambda-\mu_1\nu_1)(\lambda-\mu_2\nu_1)(\lambda-\mu_1\nu_2)(\lambda-\mu_2\nu_2).
\]
\end{proof}

\begin{remark}
When $r$ is even, the embedding $X_k(\Sigma_{g,n},\SL_2)\hookrightarrow V_P(r)$ established in Proposition~\ref{prop:chekhovmazzocco} provides a conceptual clarification of (the complexification of) the result of Chekhov--Mazzocco \cite{ChekhovMazzocco} on embeddings of Teichm\"uller spaces of surfaces into the varieties of Stokes matrices.
\end{remark}

Next, we show that the morphism $X_k(\Sigma_{g,n},\SL_2)\rightarrow V_P(r)$ is compatible with the integral structures on $X_k(\Sigma_{g,n},\SL_2)$ and $V(r)$.

\begin{proposition}
\label{prop:integralcharStokes}
The morphism $X_k(\Sigma_{g,n},\SL_2)\rightarrow V_P(r)$ defined in Proposition~\ref{prop:chekhovmazzocco} sends integral points in $X_k(\Sigma_{g,n},\SL_2)$ to integral Stokes matrices.
\end{proposition}

\begin{proof}
Let $\alpha_1,\ldots,\alpha_{r-1}$ be the hyperelliptic generators of $\pi_1(\Sigma_{g,n})$. Recall that the composition $X(\Sigma_{g,n},\SL_2)\simeq C(r,\SL_2)\simeq B(r,\SL_2)\simeq A(r,4)\rightarrow A'(r,4)$ established in previous sections is given by
\[
\rho\mapsto[(\rho(\alpha_1\cdots\alpha_{r-1}), \rho(\alpha_2\cdots\alpha_{r-1}),\ldots,\rho(\alpha_{r-1}),1)].
\]
Here we identify $\SL_2$ with $S(q_4)$, where the quadratic form is given by $q_4(x_{ij})=x_{11}x_{22}-x_{12}x_{21}=\det(x_{ij})$ on the space $V_4=\Mat_2$.
Note that for any $a,b\in\Mat_2$, we have
\begin{align*}
\left<a,b\right>_q\I_2&=\frac12\Big(\det(a+b)-\det(a)-\det(b)\Big)\I_2\\
&=\frac12\Big((a+b)(\bar a+\bar b)-a\bar a-b\bar b\Big)\\
&=\frac12\Big(a\bar b+b\bar a\Big)\\
&=\frac{\tr(a\bar b)}2\I_2
\end{align*}
By Corollary \ref{coro;A'andStokes}, the Stokes matrix associated to $\rho$ is given by
\[
s=
\begin{bmatrix}
1&\tr\rho(\alpha_1)&\tr\rho(\alpha_1\alpha_2)&\cdots&\tr\rho(\alpha_1\cdots\alpha_{r-1})\\
&1&\tr(\alpha_2)&\cdots&\tr\rho(\alpha_2\cdots\alpha_{r-1})\\
&&1&\cdots&\tr\rho(\alpha_3\cdots\alpha_{r-1})\\
&&&\ddots&\vdots\\
&&&&1
\end{bmatrix}.
\]
The proposition then follows from the fact that the integral points on $X_k(\Sigma,\SL_2)$ correspond to local systems having integral traces along every loop of $\Sigma$.
\end{proof}

Finally, we show that the morphism $X(\Sigma_{g,n},\SL_2)\rightarrow V(r)$ defined in Proposition~\ref{prop:chekhovmazzocco} is a Poisson morphism, with respect to the natural Poisson structures on $X(\Sigma_{g,n},\SL_2)$ and $V(r)$ which we now recall.

The natural Poisson structure on $X(\Sigma_{g,n},\SL_2)$ was introduced by Goldman \cite{GoldmanPoisson} for closed surfaces, and extended to surfaces with boundary for $\SL_n$-representations in a work of Lawton \cite[Theorem~15]{Lawton}.
Let $\alpha,\beta\in\pi_1(\Sigma_{g,n})$ be represented by oriented immersed curves in general position. Then the Poisson bracket of the trace functions is given by
\[
\{\tr_\alpha,\tr_\beta\}=\frac12\sum_{p\in\alpha\cap\beta}\epsilon(p;\alpha,\beta)\Big(\tr_{\alpha_p\beta_p}-\tr_{\alpha_p\beta_p^{-1}}\Big),
\]
where $\epsilon(p;\alpha,\beta)=\pm1$ denotes the oriented intersection number of $\alpha$ and $\beta$ at $p$, and $\alpha_p,\beta_p$ are elements in $\pi_1(\Sigma_{g,n},p)$ corresponding to $\alpha$ and $\beta$. Note that the symplectic leaves of this Poisson structure are the level sets of the boundary monodromy morphism $c\colon X(\Sigma_{g,n},\SL_2)\to W(\SL_2)^{\pi_0(\del \Sigma)}$.

On the other hand, there is a natural Poisson structure on $V(r)$ introduced by Dubrovin \cite{Dubrovin} and Ugaglia \cite{Ugaglia}. For $i<j$ and $k<\ell$, the Poisson bracket (up to an overall scaling) of $s_{ij}$ and $s_{k\ell}$ is given by
\[
\{s_{ij},s_{k\ell}\}=
\begin{cases}
\frac12s_{ij}s_{i\ell}-s_{j\ell} & \text{if } i=k\text{ and }j<\ell,\\
\frac12s_{ij}s_{kj}-s_{ik} & \text{if } i<k\text{ and }j=\ell,\\
s_{ij}-\frac12s_{ij}s_{j\ell} & \text{if } j=k,\\
s_{i\ell}s_{kj}-s_{ik}s_{j\ell} & \text{if } i<k<j<\ell,\\
0 & \text{if } j<k,\\
0 & \text{if } i<k \text{ and } j>\ell.
\end{cases}
\]
A description of symplectic leaves of this Poisson structures can be found in \cite[Section~5.5]{Bondal04}.

\begin{proposition}
The morphism $X(\Sigma_{g,n},\SL_2)\rightarrow V(r)$ defined in Proposition~\ref{prop:chekhovmazzocco} is a Poisson morphism.
\end{proposition}

\begin{proof}
Recall from the proof of Proposition~\ref{prop:integralcharStokes} that the morphism $X(\Sigma_{g,n},\SL_2)\rightarrow V(r)$ is given by
\[
\rho\mapsto
\begin{bmatrix}
1&\tr\rho(\alpha_1)&\tr\rho(\alpha_1\alpha_2)&\cdots&\tr\rho(\alpha_1\cdots\alpha_{r-1})\\
&1&\tr(\alpha_2)&\cdots&\tr\rho(\alpha_2\cdots\alpha_{r-1})\\
&&1&\cdots&\tr\rho(\alpha_3\cdots\alpha_{r-1})\\
&&&\ddots&\vdots\\
&&&&1
\end{bmatrix}.
\]
The proposition then follows from computing the Poisson bracket $$\{\tr_{\alpha_{i}\alpha_{i+1}\cdots\alpha_{j-1}},\tr_{\alpha_{k}\alpha_{k+1}\cdots\alpha_{\ell-1}}\}$$
for each of the six cases above.
For instance, when $i=k<j<\ell$, we have
\begin{align*}
\{\tr_{\alpha_{i}\alpha_{i+1}\cdots\alpha_{j-1}},\tr_{\alpha_{i}\alpha_{i+1}\cdots\alpha_{\ell-1}}\}&=\frac12\Big(\tr_{(\alpha_{i}\alpha_{i+1}\cdots\alpha_{j-1})^2\alpha_{j}\cdots\alpha_{\ell-1}}-\tr_{\alpha_{j}\alpha_{j+1}\cdots\alpha_{\ell-1}}\Big)\\
&=\frac12\tr_{\alpha_{i}\alpha_{i+1}\cdots\alpha_{j-1}}\tr_{\alpha_{i}\alpha_{i+1}\cdots\alpha_{\ell-1}}-\tr_{\alpha_{j}\alpha_{j+1}\cdots\alpha_{\ell-1}}.
\end{align*}
Here we used the fact that $\tr(A^2B)=\tr(AB)\tr(A)-\tr(B)$ for any $A,B\in\SL_2$.
\end{proof}

\section{Diophantine theorem}\label{sect:6}
Let $\Sigma_{g,n}$ be a surface of genus $g\geq0$ with $n\in\{1,2\}$ boundary curves, and let $r=2g+n$. By Theorem \ref{mainthm}, there is a $B_r$-invariant isomorphism
$$A_P(r,4)\simeq X_k(\Sigma_{g,n},\SL_2)$$
where the Coxeter invariant $P$ determines the boundary monodromy $k$, and vice versa. The Diophantine aspects of the latter were investigated in \cite{Whang2, Whang3}. Motivated by applications to the study of rank $4$ integral Stokes matrices, in this section we refine this Diophantine study in the case $\Sigma=\Sigma_{1,2}$ of a two-holed torus, and prove Theorem \ref{mainthm2}. This section is organized as follows. In Section \ref{sect:6.1}, we review general structure theorems for integral points on the varieties $X_k(\Sigma_{g,n},\SL_2)$. We give an analysis of the classical case $(g,n)=(1,1)$ in Section \ref{sect:6.2}, which goes back to work of Markoff \cite{markoff}. This together with preliminary observations for the case $(g,n)=(0,4)$ in Section \ref{sect:6.3} are used to prove Theorem \ref{mainthm2} in Section \ref{sect:6.4}.

\subsection{Review of structure theory}\label{sect:6.1}
Let $\Sigma$ be a compact oriented surface of genus $g$ with $n$ boundary curves satisfying $3g+n-3>0$. For $k\in\C^n$, let $X_k=X_k(\Sigma,\SL_2)$ denote the moduli of $\SL_2(\C)$-local systems on $\Sigma$ with boundary monodromy traces $k$. We will be interested in the study of integral points on $X_k$. Let us define an \emph{essential curve} in $\Sigma$ to be a simple closed curve on $\Sigma$ which is noncontractible and not isotopic to a boundary curve of $\Sigma$.

\begin{definition}
Let $k\in\C^n$. A point $\rho\in X_k(\Z)$ defined to be \emph{integral} if its monodromy trace along every essential curve on $\Sigma$ is integral.
\end{definition}

If $k\in\Z^n$, the variety $X_k$ admits a natural integral model over $\Z$, in which case the definition of $X_k(\Z)$ above coincides with the set of integral points on $X_k$ in the usual algebro-geometric sense \cite[Lemma 2.5]{Whang2}. We make the following definition.

\begin{definition}
\label{degenerate}
The \emph{degenerate locus} of $X_k$ is the union of images of nonconstant morphisms $\A^1\to X_k$. A point or a subvariety of $X_k$ is \emph{degenerate} if it belongs to the degenerate locus of $X_k$, and is \emph{nondegenerate} otherwise.
\end{definition}

It was proved in \cite{Whang} that, in the case $n\geq1$, each $X_k$ is log Calabi--Yau in the sense that it admits a normal projective compactification with trivial log canonical divisor. Definition \ref{degenerate} is motivated by consideration of the log Calabi--Yau geometry of the moduli spaces $X_k$; see \cite[Section 1.3]{Whang2} for details.

\begin{theorem}[\cite{Whang}]
\label{dioph}
The nondegenerate integral points of $X_k(\Z)$ consist of finitely many mapping class group orbits. There is a proper closed subvariety $Z\subset X_k$ whose orbit gives precisely the locus of degenerate points on $X_k$.
\end{theorem}

A modular characterization of the degenerate locus of $X_k$ was given in \cite{Whang3}, which we recall below.  A \emph{nontrivial pair of pants} in $\Sigma$ is a subsurface of genus $0$ with $3$ boundary curves each of which is either essential or a boundary curve on $\Sigma$.

\begin{theorem}
\label{deg}
A point $\rho\in X_k(\C)$ is degenerate if and only if one of the following conditions holds:
\begin{enumerate}
    \item There is an essential curve $a\subset\Sigma$ such that $\tr\rho(a)=\pm2$, or
    \item $(g,n,k)\neq(1,1,2)$ and there is a nontrivial pair of pants $\Sigma'\subset\Sigma$ such that the restriction $\rho|\Sigma'$ is reducible.
\end{enumerate}
\end{theorem}

In light of Theorem \ref{deg} above, the following special case of \cite[Corollary 5.8]{Whang} gives a stronger variant of the first part of Theorem \ref{dioph}. Given any subset $A\subset\C$, let us denote by $X_k(A)$ the subset of $X_k(\C)$ such that $\tr\rho(a)\in A$ for every essential curve $a\subset\Sigma$. The following is a corollary of \cite[Theorem 1.4]{Whang2}.

\begin{theorem}
\label{fin}
For any $k\in\C^n$, the set $X_k(\Z\setminus\{\pm2\})$
consists of finitely many mapping class group orbits.
\end{theorem}

Given a loop $\alpha$ on $\Sigma$, we shall denote by $\tr(\alpha)$ the regular function on $X_k$ defined by the traces of representations along $\alpha$: $\rho\mapsto\tr\rho(\alpha)$. 

\subsection{Case $(g,n)=(1,1)$} \label{sect:6.2}
\label{subsec:oneholedtorus}
We specialize to the case $(g,n)=(1,1)$ where $\Sigma$ is a one-holed torus. We refer to \cite[Section 2.3.1]{Whang2} for details and background. Let us fix a hyperelliptic sequence of generators $(\alpha_1,\alpha_2)$ of $\pi_1(\Sigma)$. Writing $x=\tr(\alpha_1)$, $y=\tr(\alpha_2)$, and $z=\tr(\alpha_1\alpha_2)$, it is classical (see e.g.~\cite{Goldman}) that for each $k\in\A^1$ the moduli space $X_k=X_k(\Sigma,\SL_2)$ is the affine cubic surface
\begin{align}
    x^2+y^2+z^2-xyz-2=k.
\end{align}
The mapping class group descent on $X_k$ in this case is classical and can be traced back to the 1880 work of Markoff \cite{markoff}. We record the following.

\begin{theorem}\label{11thm}
We have the following.
\begin{enumerate}
    \item We have $X_2(\Z)=\Gamma(\Sigma)\cdot\{(2,y,y),(-2,y,-y):y\in\Z\}.$
    \item We have
    $X_{-2}(\Z)=\Gamma(\Sigma)\cdot\{(0,0,0),(3,3u,3u):u\in\{\pm1\}\}.$
    \item If $k\neq2$, then $X_k(\Z)$ consists of finitely many mapping class group orbits.
\end{enumerate}
\end{theorem}

\begin{proof}
(1) First, we note that a representation $\rho\colon\pi_1(\Sigma)\to\SL_2(\C)$ is reducible if and only if $\tr\rho([\alpha_1,\alpha_2])=2$. Thus, the locus $V_2(\C)$ parametrizes the Jordan equivalence classes of reducible representations of $\pi_1(\Sigma)$. Given $\rho\in X_2(\Z)$, it suffices to show that we have $\tr\rho(a)=\pm2$ for some essential curve $a\subset\Sigma$. Let us assume without loss of generality that $\rho$ is diagonal, and write
$$\rho(\alpha_1)=\begin{bmatrix}\lambda & 0\\ 0 & \lambda^{-1}\end{bmatrix},\quad\text{and}\quad \rho(\alpha_2)=\begin{bmatrix}\mu & 0\\ 0 & \mu^{-1}\end{bmatrix}.$$
Note that $\lambda$ and $\mu$ are each either $\pm1$ or an algebraic integer of degree $2$, by our hypothesis that $\rho\in X_2(\Z)$. If $[\Q(\lambda,\mu):\Q]=4$, then the conjugates of $\lambda\mu$ are $\lambda\mu$, $\lambda^{-1}\mu$, $\lambda\mu^{-1}$, and $\lambda^{-1}\mu^{-1}$. But since $\tr\rho(\alpha_1\alpha_2)=\lambda\mu+\lambda^{-1}\mu^{-1}\in\Z$, this implies that $\lambda\mu^{-1}=\lambda\mu$ or $\lambda^{-1}\mu^{-1}$ and hence $\lambda\in\Z$ or $\mu\in\Z$, contradicting the hypothesis on degree of $\Q(\lambda,\mu)$. It follows that we must have $d=[\Q(\lambda,\mu):\Q]\leq 2$. If $d=1$, then we are done, so assume $d=2$. It is then easy to see that, up to mapping class group action (essentially equivalent to the Euclidean algorithm), there exists a nonseparating simple loop $\alpha$ such that $\rho(\alpha)=\pm\mathbb{I}$. This gives the desired result.

(2) This is classical; we briefly sketch the derivation. Let $(x,y,z)\in X_{-2}(\Z)$ be an integral solution to the equation
$$x^2+y^2+z^2-xyz=0.$$
Up to $\Gamma(\Sigma)$-action, one can assume that $(x,y,z)$ satisfies
$$|x|\leq|y|\leq|z|\leq|xy-z|\quad\text{or}\quad|x|\leq|z|\leq|y|\leq|xz-y|.$$
We will assume the first case; the second case can be treated similarly. Now, if $x=0$ then we must have $y=z=0$ from the equation for $X_{-2}$. So let us assume $x\neq0$. If $|x|\leq 2$, then the binary form $q(y,z)=y^2+z^2-xyz$ is positive semidefinite, so
$$y^2+z^2-xyz=q(y,z)=-x^2$$
admits no integral solution. So let us assume that $3\leq |x|$. From the equation defining $X_{-2}$, we have
$$|z||xy-z|=|x^2+y^2|\leq 2|y|^2.$$
If $|xy|\leq2|z|$, then
$$\frac{|xy|^2}{4}\leq |z|^2\leq 2|y|^2\implies |x|^2\leq 8,$$
a contradiction; so we must have $|xy|>2|z|$. But then
$$\frac{|xyz|}{2}< |z||xy-z|\leq 2|y|^2\implies|x|<4,$$
so that we have $x=\pm3$. Let us assume that $x=3$. From $2|z|<|xy|=3|y|$ it follows that
$$3|y|^2\leq3|yz|=9+y^2+z^2\leq 9+2|y|^2\implies |y|^2\leq 9.$$
Since $3=|x|\leq|y|$, it follows that $y=\pm3$. Substituting $(3,\pm3)$ for $(x,y)$ in the equation for $X_{-2}$ we find $z=\pm3$, with the sign of $z$ agreeing with that of $y$. Finally, the case $x=-3$ similarly leads to $(y,z)=(-3,3)$ or $(3,-3)$. A suitable composition of Dehn twists gives rise to a cyclic permutation of coordinates, so we may makes arrangements so that $x=3$.

(3) It follows by Theorem \ref{fin} (alternatively, see \cite[Section 4.2]{Whang2} for an elementary argument) that $X_k(\Z\setminus\{\pm2\})$ decomposes into finitely many mapping class group orbits. It is thus enough to show that the set of integral points $(x,y,z)\in X_k(\Z)$ satisfying $x=\pm2$ is contained in finitely many $\Gamma(\Sigma)$-orbits. For this, note first that the intersection of $X_k$ with the plane $x=2$ (resp.~$x=-2$) gives a degenerate conic
$$(y-z)^2=k-2\quad(\text{resp.~$(y+z)^2=k-2$)}$$
which is a union of two disjoint parallel lines (note $k\neq2$ by hypothesis).
Dehn twist along the curve underlying $\alpha_1$ induces an automorphism of the conic given by $(y,z)\mapsto (z,2z-y)$ (resp.~$(y,z)\mapsto(z,-2z-y)$). Under the group generated by this automorphism, any real point on the conic can be brought into a compact subset of $\R^2$. This shows that the integral points on the conic belong to finitely many $\Gamma(\Sigma)$-orbits, and we are done.
\end{proof}

\begin{remark}
If we write
$$s=\begin{bmatrix}1 & x & z\\ 0 & 1 & y\\ 0 & 0 &1\end{bmatrix}=\begin{bmatrix}1 & \tr(\alpha_1) & \tr(\alpha_1\alpha_2)\\ 0 & 1 & \tr(\alpha_2)\\ 0 & 0 &1\end{bmatrix},$$
then each $X_k$ can be embedded as a closed subvariety of the affine space of unipotent upper triangular matrices given by the following condition:
$$\det(\lambda+s^{-1}s^T)=(\lambda+1)(\lambda^2-k\lambda+1).$$
\end{remark}

\subsection{Lemma for $(g,n)=(0,4)$}\label{sect:6.3}
We refer to \cite[Section 2.3.2]{Whang2} for details and background. Let $\Sigma$ be a surface of type $(0,4)$, with boundary curves $c_1,\dots,c_4$. Here, we will be interested in $\SL_2(\C)$-local systems on $\Sigma$ whose boundary traces along $c_3$ and $c_4$ are the same. In Secttion \ref{sect:6.4}, such local systems will arise by restriction from $\SL_2(\C)$-local systems on a surface of type $(1,2)$ by cutting the surface along a nonseparating essential curve.

Fix numbers $k_1,k_2,t\in\C$, and let $k=(k_1,k_2,t,t)$. Let $(\gamma_1,\gamma_2,\gamma_3,\gamma_4)$ be the generating sequence for $\pi_1(\Sigma)$ given by simple loops around the boundary curves $(c_1,c_2,c_3,c_4)$ respectively. Let $X,Y,Z$ respectively be underlying curves of simple loops on $\Sigma$ homotopic to $\gamma_1\gamma_2$, $\gamma_2\gamma_3$, and $\gamma_1\gamma_2$, respectively, and let $x=\tr X$, $y=\tr Y$, and $z=\tr Z$. Then $X_k=X_k(\Sigma,\SL_2)$ is the affine cubic surface
$$x^2+y^2+z^2+xyz=(t^2+k_1k_2)x+t(k_1+k_2)(y+z)+4-2t^2-k_1^2-k_2^2-t^2k_1k_2.$$
(See \cite{Goldman} for details.) Given a simple closed curve $a\subset\Sigma$, we shall denote by $\tau_a\in\Gamma(\Sigma)$ the (left) Dehn twist along $a$. We record the following.

\begin{lemma}\label{dehnlem}
Assume that $k_1+k_2,k_1k_2\in\Z$, and let $d_k=(k_1-k_2)^2\in\Z$.
\begin{enumerate}
    \item Suppose $t\neq\pm2$. If $k_1\neq k_2$, then the set of integral points on the curve $\{x=2\}\subset X_k$ consists of finitely many $\langle\tau_X\rangle$-orbits. If moreover $(t^2-4)d_k$ is not a perfect square, then the said curve has no integral point.
    \item Suppose $t\neq \pm2$. If $k_1\neq -k_2$, then the set of integral points on the curve $\{x=-2\}\subset X_k$ consists of finitely many $\langle\tau_X\rangle$-orbits.
    \item Suppose $t=2u$ for some $u\in\{\pm1\}$. If $k_1,k_2\notin\{\pm2\}$, then the set of integral points on the curve $\{z=\pm2\}\subset X_k$ consists of finitely many $\langle\tau_Z\rangle$-orbits.
\end{enumerate}
\end{lemma}

\begin{proof}
(1) The curve $C=\{x=0\}\subset X_k$ can be viewed as a degenerate conic curve in $\A_{y,z}^2$ given by $F(y+z)=0$ where
$$F(W)=W^2-t(k_1+k_2)W-[2(t^2+k_1k_2)+4-2t^2-k_1^2-k_2^2-t^2k_1k_2].$$
By elementary geometry and the description of Dehn twist $\tau_X$ given in \cite[Section 2.3.2]{Whang2}, we see that $C(\Z)/\langle\tau_X\rangle$ is finite unless $C\subset\A_{y,z}^2$ is a (nonreduced) line in $\A_{y,z}^2$; this occurs precisely when $F$ has zero discriminant (cf.~proof of Theorem \ref{11thm}(3)). But we have
$$\disc(F)=t^2(k_1+k_2)^2+4[2(t^2+k_1k_2)+4-2t^2-k_1^2-k_2^2-t^2k_1k_2]=(t^2-4)d_k$$
which is nonzero by our hypotheses. If moreover $(t^2-4)d_k$ is not a perfect square, then $C(\Z)$ is empty since $F(W)=0$ admits no rational solution. The result follows.

(2) The curve $C=\{x=-2\}\subset X_k$ can be viewed as a conic curve in $\A_{y,z}^2$ given by the equation
$$(y-z)^2-t(k_1+k_2)(y-z)+[2(t^2+k_1k_2)+2t^2+k_1^2+k_2^2+t^2k_1k_2]=2t(k_1+k_2)z.$$
If $t\neq0$, then this equation defines a non-vertical and non-horizontal parabola since $k_1+k_2\neq0$ by hypothesis, whence $C(\Z)/\langle\tau_X\rangle$ is finite by elementary geometry and the description of Dehn twist $\tau_X$ given in \cite[Section 2.3.2]{Whang2}. If $t=0$, then the equation becomes
$$(y-z)^2+(k_1+k_2)^2=0$$
which admits no integer solutions since $k_1+k_2\neq0$. The result follows.

(3) Let $v\in\{\pm1\}$. The curve $C_{v}=\{z=2v\}\subset X_k$ can be viewed as a conic curve in $\A_{x,y}$ given by the equation
\begin{align*}
    x^2+y^2+2vxy&=(t^2+k_1k_2)(x+vy)+[t(k_1+k_2)-v(t^2+k_1k_2)]y\\
    &\quad +2vt(k_1+k_2)-2t^2-k_1^2-k_2^2-t^2k_1k_2.
\end{align*}
Note that $t(k_1+k_2)-v(t^2+k_1k_2)=-v(t-vk_1)(t-vk_2)\neq0$ by our hypothesis, so the above equation defines a non-vertical and non-horizontal parabola, whence $C_v(\Z)/\langle\tau_Z\rangle$ is finite by elementary geometry and the description of Dehn twist $\tau_Z$ given in \cite[Section 2.3.2]{Whang2}.
\end{proof}

\subsection{Case $(g,n)=(1,2)$ and proof of Theorem \ref{mainthm2}}\label{sect:6.4}
We specialize to the case $(g,n)=(1,2)$ where $\Sigma$ is a two-holed torus. Let us first recall the following explicit presentation of the moduli space $X_k$ given in \cite[Section 5.3]{Goldman}. Consider the presentation
$$\pi_1(\Sigma_{1,2})=\langle K_1,K_2,U,X,Y|K_1=UXY,\: K_2=UYX\rangle,$$
and define $V=UX$, $W=UY$, and $Z=XY$. Let us write the trace functions of loops by corresponding lower-case letters (e.g.~$u=\tr U$). For $k=(k_1,k_2)\in\C^2$, we have the presentation of $X_k$ as the four-dimensional subvariety of $\A_{u,v,w,x,y,z}^6$ given by
\begin{align*}
    k_1+k_2 & =yv+xw+zu-uxy\\
    k_1k_2 &=x^2+y^2+u^2+v^2+w^2+z^2-xyz-yuw-uxv+vwz-4.
\end{align*}
Let us fix a hyperelliptic sequence of generators $(\alpha_1,\alpha_2,\alpha_3)$ of $\pi_1(\Sigma)$ in such a way that $U=\alpha_1$, $X=\alpha_2^{-1}$, and $Y=\alpha_2\alpha_3$.
Let us write
\begin{align*}
s=
\begin{bmatrix}
1 & a & e & d\\
0 & 1 & b & f\\
0 & 0 & 1 & c\\
0 & 0 & 0 & 1
\end{bmatrix}=\begin{bmatrix}
1 & \tr(\alpha_1) & \tr(\alpha_1\alpha_2) & \tr(\alpha_1\alpha_2\alpha_3)\\
0 & 1 & \tr(\alpha_2) & \tr(\alpha_2\alpha_3)\\
0 & 0 & 1 & \tr(\alpha_3)\\
0 & 0 & 0 & 1
\end{bmatrix}
\end{align*}
Then in terms of the coordinates above, we have
$$u=a, \quad x=b, \quad y=f,\quad v=ab-e,\quad w=d,\quad \text{and}\quad z=c.$$
This leads to the following.

\begin{proposition}\label{prop:bdytracepoly}
For $k=(k_1,k_2)\in\A^2(\C)$, we have a presentation of $X_k$ as the subvariety of the affine space of $4\times 4$ unipotent upper triangular matrices by
\begin{align*}
k_1+k_2&=ac+bd-ef\\
    k_1k_2&=a^2+b^2+c^2+d^2+e^2+f^2-abe-adf-bcf-cde+abcd-4.
\end{align*}
Using the matrix variable $s$ above, note that we have
\begin{align*}
    \det(\lambda+s^{-1}s^T)=\lambda^4-k_1k_2\lambda^3+(k_1^2+k_2^2-2)\lambda^2-k_1k_2\lambda+1.
\end{align*}
The discriminant of the above polynomial is $\Delta_k=(k_1^2-4)^2(k_2^2-4)^2(k_1^2-k_2^2)^2$.
\end{proposition}

We will prove the following strengthening of Theorem \ref{dioph}.

\begin{theorem}\label{refinement}
Fix $k\in\A^2(\C)$ with $\Delta_k\neq0$. Then $X_k(\Sigma_{1,2},\SL_2)(\Z)$ contains at most finitely many integral $\Gamma(\Sigma)$-orbits.
\end{theorem}

\begin{proof}
Let $\rho\in X_k(\Z)$ be given. It is known by Theorem \ref{dioph} that $X_k(\Z\setminus\{\pm2\})$ decomposes into finitely many $\Gamma(\Sigma)$-orbits, so we may assume that there is an essential curve in $\Sigma$ whose trace is $\pm2$ under $\rho$. The following cases occur:
\begin{itemize}
    \item \textbf{Case I}. There is a separating essential curve $a\subset\Sigma$ such that $\tr\rho(a)=2$.
    \item \textbf{Case II.} There is a separating essential curve $a\subset\Sigma$ such that $\tr\rho(a)=-2$.
    \item \textbf{Case III.} There is no separating essential curve with trace $\pm2$ under $\rho$, but there is a nonseparating curve $a\subset\Sigma$ such that $\tr\rho(a)=\pm2$.
\end{itemize}
In the remainder of the proofs, we will treat the cases separately. Throughout the proof, let $(\alpha_1,\alpha_2,\alpha_3)$ be the hyperelliptic sequence of generators for $\pi_1(\Sigma)$ we fixed.\\

\noindent\textbf{Case I.} Suppose $a\subset\Sigma$ is a separating essential curve such that $\tr\rho(a)=2$. Let us write $\Sigma|a=\Sigma_1\sqcup\Sigma_2$ where $\Sigma_1$ is a surface of type $(1,1)$ and $\Sigma_2$ is a surface of type $(0,3)$. We may assume up to $\Gamma(\Sigma)$-action that $\Sigma_1$ is the tubular neighborhood of the union of the images of the loops $\alpha_1$ and $\alpha_2$. By Theorem \ref{11thm}(1), we may assume that $\rho\in X_k$ belongs to an algebraic curve $C_{\sigma,t}\subset X_k$ that consists of points of the form
$$\begin{bmatrix}1 & 2\sigma & t\sigma & *\\ 0 & 1 & t & * \\ 0 & 0 & 1 & *\\ 0 & 0 & 0 & 1\end{bmatrix}$$
for some $\sigma\in\{\pm1\}$ and $t\in\Z$. We claim that we must have $C_{\sigma,t}(\Z)=\emptyset$ except for $t$ in a finite subset $S\subset\Z\setminus\{\pm2\}$.

\begin{nclaim}
There is a finite set $S\subset\Z\setminus\{\pm2\}$ such that $C_{\sigma,t}(\Z)=\emptyset$ for any $t\notin S$.
\end{nclaim}

\begin{proof}[Proof of Claim]
Suppose $\rho\in C_{\sigma,t}(\Z)$. We first note that $t\neq\pm2$. Indeed, otherwise $\rho|\Sigma_1$ is abelian, whence the monodromy of $\rho$ along $b=\del\Sigma$ must be the identity (and not just of trace $2$). But then $k_1=\tr\rho(c_1)=\tr\rho(c_2)=k_2$, contradicting our hypothesis on $k$. Thus, $\rho(\alpha_2)$ must be diagonalizable. By the same argument, we see that $\rho(\alpha_1)$ cannot be $\pm\mathbb{I}$. Thus, up to global conjugation by an element of $\SL_2(\C)$, we may assume that
$$\rho(\alpha_1)=\begin{bmatrix}1 & 1\\0 & 1\end{bmatrix},\quad\rho(\alpha_2)=\begin{bmatrix}\lambda & 0\\ 0 & \lambda^{-1}\end{bmatrix},\quad\text{and}\quad\rho(\alpha_3)=\begin{bmatrix}x_{11} & x_{12}\\ x_{21} & x_{22}\end{bmatrix}$$
for some quadratic unit $\lambda\neq\pm1$ such that $\lambda+\lambda^{-1}=t$ and some $x_{ij}\in\C$. Note that
\begin{align*}
    \tr\rho(\alpha_2\alpha_3)&=\lambda x_{11}+\lambda^{-1}x_{22}\in\Z\quad\text{and}\\
    \tr\rho(\alpha_1\alpha_2\alpha_3)&=\lambda x_{11}+\lambda^{-1}x_{22}+\lambda^{-1}x_{21}\in\Z
\end{align*}
which shows that $\lambda^{-1}x_{21}\in\Z$, whence $x_{21}=m\lambda$ for some $m\in\Z$. Now, note that
$$\tr\rho(\alpha_1\alpha_3)=\tr\rho(\alpha_3)+x_{21}\in\Z$$
since $\alpha_1\alpha_3$ is (homotopic to) a simple loop whose underlying curve is essential. This shows that $k_1=\tr\rho(\alpha_3)\in\Z[\lambda]$. This implies that
$$\tr\rho(\alpha_2)^2-4= \disc(\Z[\lambda])\mid\disc(\Z[k_1])$$
which shows that $t=\tr\rho(\alpha_2)\in\Z$ belongs to a finite set, say $S\subset\Z\setminus\{\pm2\}$. This proves the claim.
\end{proof}

Thus, to prove Case I of Theorem \ref{refinement}, it suffices to show that $C_{\sigma,t}(\Z)$ decomposes into finitely many $\langle\tau_{a}\rangle$-orbits for each $t\in S$ given in the above claim. So fix $t\in S$. Let $c$ denote the simple closed curve on $\Sigma_1$ underlying the loop $\alpha_2$. Then $\Sigma|c$ is a surface of type $(0,4)$, with boundary curves
$$\del(\Sigma|c)=c_1\sqcup c_2\sqcup c_3\sqcup c_4$$
where $c_1$ and $c_2$ are the boundary curves of $\Sigma$, and the curves $c_3$ and $c_4$ correspond to $c$. Let $k'=(k_1,k_2,t,t)\in\A^4(\C)$. Pullback of representations via the immersion $\Sigma|c\to\Sigma$ induces a nonconstant morphism from the algebraic curve $C\subset X_k$ into $X_{k'}(\Sigma|c,\SL_2)$. Its image is contained in an algebraic curve of the form described in Lemma \ref{dehnlem}(1), which contains at most finitely many $\langle\tau_a\rangle$-orbits by the said lemma. Thus, $C_{\sigma,t}(\Z)$ consists of finitely many $\langle\tau_a\rangle$-orbits, as desired.\\

\noindent\textbf{Case II.} Suppose $a\subset\Sigma$ is a separating essential curve such that $\tr\rho(a)=-2$. Let $\Sigma|a=\Sigma_1\sqcup\Sigma_2$ where $\Sigma_1$ is of type $(1,1)$ and $\Sigma_2$ is of type $(0,3)$. As in the study of Case I, up to $\Gamma(\Sigma)$-action, we may assume that the surface $\Sigma_1$ is obtained by taking a closed tubular neighborhood of the union of images of the simple loops $\alpha_1$ and $\alpha_2$. By Theorem \ref{11thm}(2), we may thus assume that $\rho$ is of the form
$$\begin{bmatrix}1 & 0 & 0 & *\\ 0 & 1 & 0 & *\\ 0 & 0 & 1 & *\\ 0 & 0 & 0 & 1\end{bmatrix}\quad\text{or}\quad\begin{bmatrix}1 & 3 & 3u & *\\ 0 & 1 & 3u & *\\ 0 & 0 & 1 & *\\ 0 & 0 & 0 & 1\end{bmatrix}$$
for some $u\in\{\pm1\}$. In the first case where $\tr\rho(\alpha_1)=\tr\rho(\alpha_2)=\tr\rho(\alpha_1\alpha_2)=0$, the defining equations for $X_k$ force $k_1+k_2=0$, contradicting our hypothesis. So we may assume that $\tr\rho(\alpha_1)=3$ and $\tr\rho(\alpha_2)=\tr\rho(\alpha_1\alpha_2)=3u$.

For $u\in\{\pm1\}$, let $C_u\subset X_k(\Sigma)$ be the algebraic curve given by the equations $\tr(\alpha_1)=3$ and $\tr(\alpha_2)=\tr(\alpha_1\alpha_2)=3u$. To prove Case II of Theorem \ref{refinement}, it suffices to show that $C_u(\Z)$ for each $u\in\{\pm1\}$ consists of finitely many $\langle\tau_{a}\rangle$-orbits. Let $c$ now denote the simple closed curve on $\Sigma_1$ underlying the loop $\alpha_1$, so that $\tr\rho(c)=3$. Then $\Sigma|c$ is a surface of type $(0,4)$, with boundary curves
$$\del(\Sigma|c)=c_1\sqcup c_2\sqcup c_3\sqcup c_4$$
where $c_1$ and $c_2$ are the boundary curves of $\Sigma$, and the curves $c_3$ and $c_4$ correspond to $c$. Let $k'=(k_1,k_2,3,3)\in\A^4(\C).$
Pullback of representations under the immersion $\Sigma|c\to\Sigma$ induces a nonconstant morphism from the algebraic curve $C_u\subset X_k$ into $X_{k'}(\Sigma|c,\SL_2)$. Its image is contained in an algebraic curve of the form described in Lemma \ref{dehnlem}(2), which contains at most finitely many $\langle\tau_a\rangle$-orbits by the said lemma. Thus, $C_u(\Z)$ consists of finitely many $\langle\tau_a\rangle$-orbits, as desired.\\

\noindent\textbf{Case III.} Suppose $a\subset\Sigma$ is a nonseparating curve such that $\tr\rho(a)=\pm2$, and suppose that no separating essential curve on $\Sigma$ has trace $\pm2$ under $\rho$. Up to $\Gamma(\Sigma)$-action, we may assume that $a$ is the curve underlying the simple loop $\alpha_1$.

We claim that the restriction $\rho|(\Sigma|a)$ belongs to one of finitely many $\Gamma(\Sigma|a)$-orbits in $X(\Sigma|a,\SL_2)$ determined by the boundary condition $k=(k_1,k_2)$. Indeed, we are done by Theorem \ref{fin} if there is no essential curve on $c\subset\Sigma|a$ such that $\tr\rho(a)=\pm2$, so let us assume that there exist such a $c$. Since no separating essential curve on $\Sigma$ has trace $\pm2$ under $\rho$ by assumption, we may assume that the image of $c$ in $\Sigma$ is nonseparating. By choosing the generating loops of $\pi_1(\Sigma|a)$ judiciously, we may assume that $\rho|(\Sigma|a)$ belongs into an algebraic curve of the form described in Lemma \ref{dehnlem}(3). By Lemma \ref{dehnlem}(3), it follows that $\rho|(\Sigma|a)$ belongs to one of finitely many $\Gamma(\Sigma|a)$-orbits in $X(\Sigma|a,\SL_2)$ determined by $k=(k_1,k_2)$.

Thus, let us assume without loss of generality that $\rho|(\Sigma|a)$ is one of finitely many points in $X(\Sigma|a,\SL_2)$. Let $b$ be an essential curve in $\Sigma|a$ whose image in $\Sigma$ is a separating essential curve. By our assumption, there is a constant $K=K(k_1,k_2)$ such that $|\tr\rho(b)|\leq K$. Moreover, by our assumption on $\rho$, we have $\tr\rho(b)\neq\pm2$. Let us write
$$(\Sigma|a)|b=\Sigma_1'\sqcup\Sigma_2$$
where $\Sigma_1'$ and $\Sigma_2$ are each a surface of type $(0,3)$, such that the boundary curves of $\Sigma_1'$ map onto $a$ and $b$ under the immersion $\Sigma_1'\to\Sigma$ while the boundary curves of $\Sigma_2$ map to $b$ and $\del\Sigma=c_1\sqcup c_2$. We remark that $\rho|\Sigma_1'$ must be irreducible, seeing as the monodromy traces of $\rho|\Sigma_1'$ along two of the boundary curves (corresponding to $a$) are both $\pm2$ while the trace along the remaining one (corresponding to $b$) is not. It follows \emph{a fortiori} that $\rho|(\Sigma|a)$ is irreducible. Let $\Sigma_1$ denote the component of $\Sigma|b$ that is of type $(1,1)$. It follows, by the paragraph on nonseparating curves in \cite[Section 2.2.3]{Whang2}, that the locus
$$C=\{\rho'\in X_k:\rho'|(\Sigma|a)=\rho|(\Sigma|a)\}\subset X_k$$
is an algebraic curve whose image under the restriction morphism $X_k\to X(\Sigma_1)$ is nonconstant. It follows by arguing as in the proof of Theorem \ref{11thm}(3) that this image of $C$ in $X(\Sigma_1)$ consists of at most finitely many $\langle\tau_a\rangle$-orbits. Thus, \emph{a fortiori} the curve $C$ has finitely many $\langle\tau_a\rangle$-orbits, whence $\rho$ belongs to the $\Gamma(\Sigma)$-orbit of one of finitely many points in $X_k(\Z)$, as desired.

This completes the proof of Theorem \ref{refinement}.
\end{proof}

We now prove Theorem~\ref{mainthm2}.

\begin{proof}[Proof of Theorem \ref{mainthm2}]
Let $p(\lambda)\in\Z[\lambda]$ be a monic reciprocal polynomial of degree $4$ with $\disc(p)\neq0$. We would like to show that the $B_4$-invariant subvariety
$$V_p(4)=\{s\in V(4):\det(\lambda+s^{-1}s^T)=p(\lambda)\}\subset V(r)$$
contains at most finitely many integral $B_4$-orbits. By Theorem~\ref{mainthm}, Proposition~\ref{prop:chekhovmazzocco} and \ref{prop:bdytracepoly}, the subvariety $V_p(4)$ is isomorphic to the disjoint union
\[
V_p(4)\cong\coprod_k X_k(\Sigma_{1,2},\SL_2),
\]
where $k=(k_1,k_2)\in\A^2(\C)$ are such that
\[
\lambda^4-k_1k_2\lambda^3+(k_1^2+k_2^2-2)\lambda^2-k_1k_2\lambda+1=p(\lambda).
\]
By Theorem~\ref{refinement} and the assumption that $\disc(p)\neq0$, each $X_k(\Sigma_{1,2},\SL_2)(\Z)$ contains at most finitely many integral $\Gamma(\Sigma)$-orbits. Theorem~\ref{mainthm2} then follows from the compatibility of integral structures on $X_k(\Sigma_{1,2},\SL_2)$ and $V_p(4)$ proved in Proposition~\ref{prop:integralcharStokes}.
\end{proof}

\section{Exceptional collections} \label{sect:7}
In this section, we apply the isomorphism $A_P(r,4)\simeq X_k(\Sigma_{g,n},\SL_2)$ and the structure results of integral points on $X_k(\Sigma_{g,n},\SL_2)$ established in previous sections, to obtain the finiteness of possible Gram matrices of full exceptional collections (up to mutations) of certain triangulated categories.

We start with recalling some of the theory of exceptional collections developed by Bondal, Gorodentsev, Polishchuk, Rudakov, and many others. The reader is referred to the original papers \cite{Bondal90,BondalPolishchuk,GorRud} for more details. 

Let $\Dcal$ be a triangulated category. An object $E\in\Dcal$ is called \emph{exceptional} if
\[
\Hom_\Dcal^0(E,E)=\C \text{ \ and \ } \Hom_\Dcal^k(E,E)=0 \text{ for all } k\in\Z\backslash\{0\}.
\]
An ordered collection of exceptional objects $\{E_1,\ldots,E_r\}$ is called an \emph{exceptional collection} if for any $r\geq i>j\geq1$,
\[
\Hom_\Dcal^k(E_i,E_j)=0 \text{ for all } k\in\Z.
\]
An exceptional collection $\{E_1,\ldots,E_r\}$ is called \emph{full} if for any object $E\in\Dcal$,
\[
\Hom_\Dcal^k(E_i,E)=0 \text{ for all } 1\leq i\leq r \text{ and all } k\in\Z \implies E\simeq0.
\]

Given two objects $E$ and $F$ of $\Dcal$, one defines objects $\L_EF$ and $\R_FE$ of $\Dcal$ (called \emph{left} and \emph{right mutation}, respectively) by the distinguished triangles
\[
\L_EF\rightarrow\Hom^\bullet_\Dcal(E,F)\otimes E\rightarrow F \text{ \ and \ }
E\rightarrow\Hom^\bullet_\Dcal(E,F)^*\otimes F\rightarrow\R_FE.
\]
A mutation of an exceptional collection $\Ecal=\{E_1,\ldots,E_r\}$ is defined via a mutation of a pair of adjacent objects in the collection as follows:
\[
\L_i\Ecal\coloneqq\{E_1,\ldots,E_{i-1},\L_{E_i}E_{i+1},E_i,E_{i+2},\ldots,E_r\},
\]
\[
\R_i\Ecal\coloneqq\{E_1,\ldots,E_{i-1},E_{i+1},\R_{E_{i+1}}E_i,E_{i+2},\ldots,E_r\}.
\]

\begin{theorem}[\cite{Bondal90,GorRud}]
A mutation of an exceptional collection is again an exceptional collection. Moreover, the following relations hold:
\[
\L_i\R_i=\R_i\L_i=\id,\ \ \L_i\L_{i+1}\L_i=\L_{i+1}\L_i\L_{i+1},\ \ \R_i\R_{i+1}\R_i=\R_{i+1}\R_i\R_{i+1},
\]
\[
\L_i\L_j=\L_j\L_i \text{ and } \R_i\R_j=\R_j\R_i \text{ if } |i-j|\neq1.
\]
Hence the braid group $B_r$ acts on the set of exceptional collections of length $r$ in $\Dcal$ by left (or right) mutations.
\end{theorem}

Let $\Ecal=\{E_1,\ldots,E_r\}$ be an exceptional collection in $\Dcal$. Denote
\[
s_\Ecal\coloneqq\Big(\chi(E_i,E_j)\Big)_{1\leq i,j\leq r}
\]
the Gram matrix of $\Ecal$ with respect to the Euler pairing
\[
\chi(E,F)\coloneqq\sum_{k\in\Zb}(-1)^k\dim\Hom_\Dcal^k(E,F).
\]
Observe that $s_\Ecal\in V(r)$ is an integral unipotent upper triangular matrix. Mutations of exceptional collections act on the Gram matrices in the following way:
\[
s_{\L_i\Ecal}=
\begin{bmatrix}
\I_{i-1} & & & \\
& s_{i,i+1} & -1 & \\
& 1 & 0 & \\
& & & \I_{r-i-1}
\end{bmatrix}
\cdot s_\Ecal\cdot
\begin{bmatrix}
\I_{i-1} & & & \\
& s_{i,i+1} & 1 & \\
& -1 & 0 & \\
& & & \I_{r-i-1}
\end{bmatrix},
\]
\[
s_{\R_i\Ecal}=
\begin{bmatrix}
\I_{i-1} & & & \\
& 0 & 1 & \\
& -1 & s_{i,i+1} & \\
& & & \I_{r-i-1}
\end{bmatrix}
\cdot s_\Ecal\cdot
\begin{bmatrix}
\I_{i-1} & & & \\
& 0 & -1  & \\
& 1 & s_{i,i+1}  & \\
& & & \I_{r-i-1}
\end{bmatrix}.
\]
Note that the actions on the Gram matrices by left mutations are compatible with the braid group actions on the Stokes matrices defined in Section~\ref{sect:5.1}.

We recall a well-known relationship between the Serre functor and the Coxeter element associated to the Gram matrix of any full exceptional collection.

\begin{lemma}
\label{lemma:serrecoxeter}
Let $\Dcal$ be a triangulated category that admits a Serre functor $\Sb_\Dcal$ and a full exceptional collection of length $r$. Then there exists a monic reciprocal polynomial $p$ of degree $r$ such that $s_\Ecal\in V_p(r)$ for any full exceptional collection $\Ecal$ of $\Dcal$.
\end{lemma}

\begin{proof}
Let $\Ecal=\{E_1,\ldots,E_r\}$ be a full exceptional collection of $\Dcal$. Then the classes $\{[E_1],\ldots,[E_r]\}\subset K_0^\num(\Dcal)$ form a basis of the numerical Grothendieck group $K_0^\num(\Dcal)$. It is easy to check that if we consider $v\in K_0^\num(\Dcal)$ as a column vector with respect to the basis $\{[E_1],\ldots,[E_r]\}$, then the induced automorphism $\Sb_\Dcal^\num\colon K_0^\num(\Dcal)\rightarrow K_0^\num(\Dcal)$ can be written as
\[
\Sb_\Dcal^\num(v)=s_\Ecal^{-1}s_\Ecal^Tv.
\]
Define the polynomial
\[
p(\lambda)\coloneqq\det(\lambda+\Sb_\Dcal^\num).
\]
Then it is clear that $p$ is a monic reciprocal polynomial of degree $r$, and $s_\Ecal\in V_p(r)$ for any full exceptional collection $\Ecal$.
\end{proof}

\begin{remark}
In the context of Fukaya--Seidel category of a Lefschetz fibration, the Coxeter identity in Proposition~\ref{proposition:coxeterStokes} is reminiscent of the relationship between the global monodromy and the Serre functor. We refer to \cite{Seidel} for more details.
\end{remark}

\begin{example}
\label{eg:derivedcategory}
Let $\Dcal=\Dcal^b\Coh(X)$ be the bounded derived category of coherent sheaves on a smooth projective variety $X$ of dimension $n$. The Serre functor is given by $\Sb_\Dcal=(-\otimes K_X)[n]$, where $K_X$ denotes the canonical bundle on $X$. Since $(-\otimes K_X)$ induces a unipotent operator on the Grothendieck group of $\Dcal$ (\cite[Lemma~3.1]{BondalPolishchuk}), the Serre operator $\Sb_\Dcal^\num$ satisfies the property that $(-1)^n\Sb_\Dcal^\num$ is unipotent.
Suppose that $\Dcal$ admits a full exceptional collection $\Ecal=\{E_1,\ldots,E_r\}$. By Lemma~\ref{lemma:serrecoxeter}, the Gram matrix $s_\Ecal$ satisfies the following properties:
\begin{itemize}
    \item $s$ is a unipotent upper triangular $r\times r$ matrix,
    \item $(-1)^ns^{-1}s^T$ is unipotent.
\end{itemize}
\end{example}

Let $\Dcal$ be a triangulated category that admits a full exceptional collection. It is interesting to understand all possible Gram matrices of full exceptional collections of $\Dcal$, up to the natural $B_r$-actions. Using the $B_r$-equivariant isomorphisms between $A_P(r,4)$ and the $\SL_2$-character variety, together with Diophantine results from Section \ref{sect:6},
we are able to establish finiteness result for Gram matrices of \emph{nondegenerate} full exceptional collections.

\begin{definition}
Let $\Ecal$ be an exceptional collection of length $r$ of $\Dcal$ and let $p$ be the characteristic polynomial of the matrix $-s_\Ecal^{-1}s_\Ecal^T$. We say $\Ecal$ is \emph{degenerate} if its Gram matrix $s_\Ecal\in V_p(r)$ lies in the image of a nonconstant morphism $\A^1\rightarrow V_p(r)$ from the affine line, and is said to be \emph{nondegenerate} otherwise.
\end{definition}

\begin{remark}
It would be interesting to give a categorical characterization of degenerate exceptional collections.
\end{remark}

\begin{theorem}
\label{thm:excepapp}
Let $\Dcal$ be a triangulated category admitting a full exceptional collection of length $4$ and a Serre functor $\Sb_\Dcal$. Then there is a finite list of integral Stokes matrices of rank $4$ such that, up to mutations, the Gram matrix of any \emph{nondegenerate} full exceptional collection of $\Dcal$ belongs to this list.

Moreover, if the discriminant of the polynomial $p(\lambda)=\det(\lambda+\Sb_\Dcal^\num)$ is nonzero,
then there is a finite list of integral Stokes matrices of rank four such that, up to mutations, the Gram matrix of any full exceptional collection of $\Dcal$ belongs to this list.
\end{theorem}

\begin{proof}
By Lemma~\ref{lemma:serrecoxeter}, there exists a reciprocal polynomial $p$ of degree $4$ such that $s_\Ecal\in V_p(4)$ for any full exceptional collection $\Ecal$ of $\Dcal$.
Recall from Example~\ref{r4example} that $V_p(4)$ is isomorphic to a disjoint union of $B_4$-invariant varieties of the form $A_P(4,4)$ for some $P\in \Spin(4)\git\SO(4)$. By Theorem~\ref{mainthm}, for each $A_P(4,4)$ there is a $B_4$-equivariant isomorphism 
\[
A_P(4,4)\simeq X_k(\Sigma_{1,2},\SL_2)
\]
for some $k\in\C^2$.
Since the morphisms $X_k(\Sigma_{1,2},\SL_2)\to V_p(4)$ send integral points to integral points (Proposition~\ref{prop:integralcharStokes}), the theorem then follows from the Diophantine results on character varieties Theorems~\ref{dioph} and \ref{refinement}.
\end{proof}

We conclude the section with a few remarks.

\begin{remark}
Let $\Dcal=\Dcal^b\Coh(X)$ be the derived category of an algebraic surface $X$ admitting a full exceptional collection $\Ecal$ of length four. Then the Gram matrix $s_\Ecal$ has the property that $s_\Ecal^{-1}s_\Ecal^T$ is unipotent by Example~\ref{eg:derivedcategory}, and hence $s_{\Ecal}\in V_p(4)$ where $p(\lambda)=(\lambda-1)^4$. In terms of the matrix coefficients
\[
s=
\begin{bmatrix}
1 & a & e & d\\
0 & 1 & b & f\\
0 & 0 & 1 & c\\
0 & 0 & 0 & 1
\end{bmatrix}
\]
the variety $V_p(4)$ is given by
\begin{equation}
\label{equation:rank4unipotent}
\begin{cases}
a^2+b^2+c^2+d^2+e^2+f^2-abe-adf-bcf-cde+abcd=0,\\
ac+bd-ef=0.
\end{cases}
\end{equation}
The solutions to this system of Diophantine equations have been studied in \cite{dTdVdB}. It is proved in \cite[Theorem~A]{dTdVdB} that any integral solution to (\ref{equation:rank4unipotent}) is equivalent to one of the following solutions under the signed braid group actions:
\[
\begin{bmatrix}1&2&2&4\\0&1&0&2\\0&0&1&2\\0&0&0&1\end{bmatrix},
\text{ \ or \ }
\begin{bmatrix}1&n&2n&n\\0&1&3&3\\0&0&1&3\\0&0&0&1\end{bmatrix}
\text{ \ for \ }
n\in\N.
\]
Note that this does not contradict with the conclusion of Theorem~\ref{thm:excepapp}, since both of these types of matrices are degenerate points in $V_p(4)$ and $\Delta_p=0$.
\end{remark}

\begin{remark}
In general, the Gram matrix of an exceptional collection of length $r>4$ does not lie in the image of the composition
$$X_k(\Sigma_{g,n},\SL_2)\simeq A_P(r,4)\rightarrow A'_{P'}(r,4)\hookrightarrow V(r,4)\hookrightarrow V(r).$$
For instance, suppose that $X$ is an even dimensional smooth projective variety such that $\Dcal^b\Coh(X)$ admits a full exceptional collection $\Ecal$. Then $s_\Ecal^{-1}s_\Ecal^T$ is unipotent, hence $s_\Ecal+s_\Ecal^T$ is invertible. On the other hand, if $s$ is in the image of the above composition, then $s+s^T$ is not invertible when $r>4$.
\end{remark}

\begin{thebibliography}{99}

\bibitem{BirHil1}
J.~S.~Birman and H.~M.~Hilden.
Lifting and projecting homeomorphisms.
\emph{Arch. Math. (Basel) 23},
428--434,  1972.

\bibitem{BirHil2}
J.~S.~Birman and H.~M.~Hilden.
On isotopies of homeomorphisms of Riemann surfaces.
\emph{Ann. of Math. (2)},
97:424--439, 1973.


\bibitem{Bondal90}
A.~I.~Bondal.
Representations of associative algebras and coherent sheaves.
\emph{Math. USSR-Izv.},
34(1):23--42, 1990.

\bibitem{Bondal04}
A.~I.~Bondal.
A symplectic groupoid of triangular bilinear forms and the braid group.
\emph{Izv. Ross. Akad. Nauk Ser. Mat.},
68(4):19--74, 2004.

\bibitem{BondalPolishchuk}
A.~I.~Bondal and A.~E.~Polishchuk.
Homological properties of associative algebras: The method of helices.
\emph{Russian Acad. Sci. Izv. Math.},
42(2):219--260, 1994.

\bibitem{Bourbaki}
N.~Bourbaki.
Groupes et alg\`ebres de Lie, Chap. 4,5 et 6.
Masson, Paris--New York--Barcelone--Milan--Mexico--Rio de Janeiro, 1981.

\bibitem{bries}
E.~Brieskorn.
Automorphic sets and braids and singularities.
\emph{Braids (Santa Cruz, CA, 1986)}, 45--115,
Contemp. Math., 78,
\emph{Amer. Math. Soc., Providence, RI}, 1988.

\bibitem{ChekhovMazzocco}
L.~Chekhov and M.~Mazzocco.
Teichm\"uller spaces as degenerated symplectic leaves in Dubrovin--Ugaglia Poisson manifolds.
\emph{Phys. D},
241(23-24):2109--2121, 2012.

\bibitem{dTdVdB}
L.~de~Thanhoffer~de~Volcsey and M.~Van~den~Bergh.
On an analogue of the Markov equation for exceptional collections of length 4, 2016.
arXiv:1607.04246.

\bibitem{Dubrovin}
B.~Dubrovin.
Geometry of 2D topological field theories.
Lecture Notes in Math., 1620, \emph{Springer, Berlin}, 1996.

\bibitem{fm}
B.~Farb and  D.~Margalit.
A primer on mapping class groups.
Princeton Mathematical Series, 49. Princeton University Press, Princeton, NJ, 2012.

\bibitem{GoldmanPoisson}
W.~M.~Goldman.
Invariant functions on Lie groups and Hamiltonian flows
of surface group representations.
\emph{Invent. Math.},
85(2):263--302, 1986.

\bibitem{Goldman}
W.~M.~Goldman.
Trace Coordinates on Fricke spaces of some simple hyperbolic surfaces.
\emph{Handbook of Teichmüller theory. Vol. II}, 611--684,
IRMA Lect. Math. Theor. Phys., 13,
\emph{Eur. Math. Soc., Z\"urich}, 2009.

\bibitem{GorRud}
A.~L.~Gorodentsev and A.~N.Rudakov.
Exceptional vector bundles on projective spaces.
\emph{Duke Math. J.},
54(1):115--130, 1987.

\bibitem{joyce}
D.~Joyce.
A classifying invariant of knots, the knot quandle.
\emph{J. Pure Appl. Algebra} 23(1):37--65, 1982.

\bibitem{Lawton}
S.~Lawton.
Poisson geometry of $\SL(3,\C)$-character varieties relative to a surface with boundary.
\emph{Trans. Amer. Math. Soc.},
361(5):2397--2429, 2009.

\bibitem{markoff}
A.~Markoff.
Sur les formes quadratiques binaires ind\'efinies.
\emph{Math. Ann.}, 17:379--399, 1880.

\bibitem{Procesi}
C.~Procesi.
Lie groups: an approach through invariants and representations.
\emph{Springer, New York}, 2007.

\bibitem{Seidel}
P.~Seidel.
Symplectic homology as Hochschild homology.
\emph{Algebraic geometry—Seattle 2005. Part 1}, 415--434,
Proc. Sympos. Pure Math., 80, Part 1,
\emph{Amer. Math. Soc., Providence, RI}, 2009.

\bibitem{Ugaglia}
M.~Ugaglia.
On a Poisson structure on the space of Stokes matrices.
\emph{Int. Math. Res. Not.},
(9):473--493, 1999.

\bibitem{Weyl}
H.~Weyl.
The Classical Groups. Their Invariants and Representations.
\emph{Princeton University Press, Princeton, N.J.}, 1939.

\bibitem{Whang}
J.~P.~Whang.
Global geometry on moduli of local systems for surfaces with boundary, 2016.
To appear in \emph{Compositio Mathematica}.
arXiv:1612.02518.

\bibitem{Whang2}
J.~P.~Whang.
Nonlinear descent on moduli of local systems, 2017.
To appear in \emph{Israel Journal of Mathematics}.
arXiv:1710.01848.

\bibitem{Whang3}
J.~P.~Whang.
Arithmetic of curves on moduli of local systems, 2018.
To appear in \emph{Algebra \& Number Theory}.
arXiv:1803.04583.
\end{thebibliography}
\end{document}